\documentclass[reqno]{amsart}

\usepackage{amsmath,amssymb,mathtools,amsopn,amsfonts,amsthm}
\usepackage{bbm}
\usepackage{subfigure}
\usepackage{eepic}
\usepackage{a4wide}
\usepackage[utf8]{inputenc}
\usepackage{epsfig}
\usepackage{latexsym}
\usepackage{enumerate, enumitem}
\usepackage[UKenglish]{babel}
\usepackage{verbatim}
\usepackage[initials]{amsrefs}
\usepackage[bookmarks]{hyperref}

\usepackage{pgf,tikz,pgfplots}
\usepackage{mathrsfs}
\usepackage{graphicx}
\graphicspath{{tikz pics/}}
\usepackage{grffile}
\usetikzlibrary{arrows.meta, calc,patterns,angles,quotes}
\usetikzlibrary{decorations.pathreplacing,decorations.markings}
\tikzstyle arrowstyle=[scale=1.5]
\tikzstyle directed=[postaction={decorate,decoration={markings,
  mark=at position 0.48 with {\arrow[arrowstyle]{{To}}}}}]
\definecolor{xcol}{RGB}{59,100,179}
\definecolor{ycol}{RGB}{255,85,85}

\newtheorem{theorem}{Theorem}[section]
\newtheorem{lemma}[theorem]{Lemma}
\newtheorem{corollary}[theorem]{Corollary}

\theoremstyle{definition}
\newtheorem{definition}[theorem]{Definition}

\newtheorem{example}[theorem]{Example}
\newtheorem{question}{Question}

\newcommand{\R}{\mathbb{R}}
\newcommand{\tr}{\textnormal{tr}}

\newcommand{\N}{\mathbb{N}}

\newcommand{\Z}{\mathbb{Z}}
\newcommand{\C}{\mathbb{C}}

\newcommand{\SL}{\mathrm{SL}(2,\mathbb{R})}
\renewcommand{\H}{\mathbb{H}}
\renewcommand{\P}{\mathbb{P}}
\providecommand{\norm}[1]{\lVert#1\rVert}

\newcommand{\A}{\mathcal{A}}

\let\Im\relax
\DeclareMathOperator{\Im}{Im}

\DeclareMathOperator{\trace}{tr}

\renewcommand{\epsilon}{\varepsilon}
\renewcommand{\geq}{\geqslant}
\renewcommand{\leq}{\leqslant}
\renewcommand{\subset}{\subseteq}

\newcommand{\F}{\mathcal{F}}
\newcommand{\HY}{\mathcal{H}}
\newcommand{\EL}{\mathcal{E}}
\newcommand{\SD}{\mathcal{S}}
\newcommand{\LF}{\Lambda^{+}}
\newcommand{\LB}{\Lambda^{-}}
\newcommand{\PSL}{\mathrm{PSL}(2,\mathbb{R})}
\newcommand{\SF}{\langle \F \rangle}
\newcommand{\SDP}{\SD\setminus \overline{H_P}}

\title{Parameter spaces of locally constant cocycles}

\author{Argyrios Christodoulou} \address{School of Mathematical Sciences, Queen Mary University of London, London E1 4NSH}
\email{argyrios.christodoulou@qmul.ac.uk}

\begin{document}

\maketitle

\begin{abstract}
This article concerns the locus of all locally constant $\mathrm{SL}(2,\mathbb{R})$-valued cocycles that are uniformly hyperbolic, called the \emph{hyperbolic locus}. Using the theory of semigroups of M\"obius transformations we introduce a new locus in $\mathrm{SL}(2,\mathbb{R})^N$ which allows us to study the complement of the hyperbolic locus. Our results answer a question of Avila, Bochi and Yoccoz, and Jacques and Short, while motivating a new line of investigation on the subject.
\end{abstract}

\section{Introduction}
Let $X$ be a compact metric space and $T\colon X\to X$ a homeomorphism. A \emph{linear cocycle} over $T$ is a bundle map 
\begin{align*}
F\colon X\times\R^d&\to X\times\R^d\\
(x,u)&\mapsto(T(x), \A(x)u),
\end{align*}
where $\A\colon X\to \mathrm{GL}(d,\mathbb{R})$ is a continuous function. Linear cocycles appear naturally in the study of dynamical systems as they provide models for studying various properties of the map $T$. For example, when $X$ is a compact manifold and $T$ is a diffeomorphism, the derivative of $T$ can be studied through a linear cocycle called the \emph{derivative cocycle} which plays a fundamental role in determining the behaviour of the dynamical system $T$ (see, for example, \cite[Section 5]{BaPe2007} or \cite[Section 2.1.2]{Vi2014}). We note a few articles from the broad literature that illustrate the connection of linear cocycles with various aspects of dynamical systems, such as partially hyperbolic \cite{ASVW2013,AvVi2010} and hyperbolic \cite{Ka2011,Vi2008} dynamics, Anosov representations \cite{BoPoSa2019} and Teichm\"uller flows \cite{Fo2006}.

The study of general cocycles can prove quite challenging, even when the function $\A$ is H\"older continuous. Two conditions, commonly found in the literature, that alleviate some of the difficulties of the general case are \emph{fiber-bunching} and \emph{domination}. In crude terms, a cocycle $F$ is fiber-bunched when the fiber maps $F(x,\cdot)$ are close to being conformal, while $F$ is dominated if the vector bundle $X\times\R^d$ can be split into a direct sum of $F$-invariant subbundles such that $F$ expands one subbundle more that the other. The ideas behind these concepts can be traced back to the 1980's, but we refer to \cite{BoGa2019} and \cite{BoGo2009}, and references therein, for a modern analysis of fiber-bunching and domination, respectively. 

In this article we continue the study of a problem regarding dominated cocycles that was considered by Avila, Bochi and Yoccoz \cite{AvBoYo2010} and Yoccoz \cite{Yo2004}: how can we describe the set of all dominated, locally constant $\SL$ cocycles?

So for the rest of this article we restrict ourselves to the following setting:\\
Fix a positive integer $N>1$ and let $\Sigma=N^\Z$ denote the full two-sided shift on $N$ symbols equipped with the shift map $\sigma\colon \Sigma\to\Sigma$. Consider an $\SL$-valued linear cocycle over $\sigma$
\begin{align*}
F\colon \Sigma\times\R^2&\to \Sigma\times\R^2\\
(x,u)&\mapsto(\sigma x, \A(x)u),
\end{align*}
where $\A\colon \Sigma\to \mathrm{SL}(2,\mathbb{R})$ is a continuous function. By writing $\A^n(x)=\A(\sigma^{n-1}x)\A(\sigma^{n-2}x)\cdots \A(x)$, we obtain that $F^n(x,u)=(\sigma^nx,\A^n(x))$ for all $n\in\N$. The cocycle $F$ is called \emph{locally constant} if $\A(x)$ depends only on the letter $x_0$ in position zero of $x=(x_n)\in\Sigma$. Yoccoz \cite[Proposition 2]{Yo2004} showed that a locally constant $\SL$-cocycle is dominated if and only if there exist constants $\lambda>1$ and $C>0$, such that 
\[
\norm{\A^n(x)}\geq C\lambda^n, \quad \text{for all}\ x\in\Sigma\ \text{and all}\ n\geq0.
\]
Dominated locally constant $\SL$-cocycles are often called \emph{uniformly hyperbolic} and we adopt this terminology from now on.

In essence, a locally constant cocycle depends only on a collection of $N$ matrices from $\SL$. So, we can consider the parameter space $\SL^N$ with the natural product topology, where each $N$-tuple gives rise to a locally constant cocycle in the obvious way. The definition of uniform hyperbolicity can now be restated as follows.

\begin{definition}\label{hy}
An $N$-tuple $(A_1,A_2,\dots,A_N)\in\SL^N$ is called \emph{uniformly hyperbolic} if there exist constants $\lambda>1$ and $C>0$, such that 
\[
\norm{A_{i_n}A_{i_{n-1}}\dots A_{i_0}}\geq C\lambda^n, \quad \text{for any sequence}\ (i_n)\subset \{1,2,\dots,N\} \ \text{and all}\ n\in\N.
\]
The collection of all uniformly hyperbolic points in $\SL^N$ is called \emph{the hyperbolic locus}.
\end{definition}

It is interesting to note that $(A_1,A_2,\dots,A_N)\in\SL^N$ is uniformly hyperbolic if and only if the \emph{lower spectral radius} (or \emph{joint spectral subradius}) of the set $\mathbb{A}=\{A_1,A_2,\dots,A_N\}\subseteq\SL$ defined by
\[
\check{\rho}(\mathbb{A})\vcentcolon= \lim_{n\to\infty}\inf \left\{\norm{A_{i_n} A_{i_{n-1}} \dots A_{i_1}}^{1/n} \colon A_{i_j}\in \mathbb{A}\right\},
\]
is strictly greater than one. The lower spectral radius was introduced in \cite{Gu1995}. We refer to \cite{BoMo2015} and \cite{Ju2012} for an interesting exploration of its continuity properties that are related to uniform hyperbolicity, or domination in general.

The hyperbolic locus was introduced by Yoccoz \cite{Yo2004} and further investigated in \cite{AvBoYo2010}. It is easy to see that the hyperbolic locus is open, but not connected. For $N=2$, the authors of \cite{AvBoYo2010} give a full description of the components of the hyperbolic locus, but note that for $N\geq 3$ ``new phenomena appear, which make such a complete description much more difficult and complicated". A set that played an important role their analysis is the following, where we recall that a non-identity matrix $A\in\mathrm{SL}(2,\R)$ is called \emph{elliptic} if $\lvert \trace{A}\rvert<2$, \emph{parabolic} if $\lvert \trace{A}\rvert=2$ and \emph{hyperbolic} if $\lvert \trace{A}\rvert>2$.

\begin{definition}\label{el}
We define $\EL$ to be the set of $N$-tuples $(A_1,A_2,\dots, A_N)$ in $\mathrm{SL}(2,\mathbb{R})^N$ for which the semigroup of matrices generated by $\{A_1,A_2,\cdots,A_N\}$ contains either the identity or an elliptic matrix. The locus $\EL$ is called the \emph{elliptic locus}.
\end{definition}

By definition, the sets $\HY$ and $\EL$ are disjoint. In fact, Yoccoz \cite[Proposition 6]{Yo2004}, with a proof credited to Avila, showed that $\overline{\EL}=\HY^c$ in $\SL^N$. Avila, Bochi and Yoccoz \cite[Theorem 3.3]{AvBoYo2010} show that when $N=2$ the sets $\HY$ and $\EL$ share the same boundary i.e. $\overline{\HY}=\EL^c$, and ask whether the same holds for any positive integer $N$ (see \cite[Question 4]{AvBoYo2010} and \cite[Question 4]{Yo2004}):

\begin{question}\label{Q}
Is it true that $\overline{\HY}={\EL}^c$ for any positive integer $N$?
\end{question}

Our main goal in this article is to show that for any $N\geq3$ the answer to Question~\ref{Q} is negative. This will be carried out in Section~\ref{counter}. For a similar question on more general linear cocycles, see \cite[Question 7.4]{BaKaKo2018}.

We note that the elliptic locus was originally defined in \cite{Yo2004} and \cite{AvBoYo2010} as the set of $N$-tuples whose semigroup contains an elliptic matrix (and not the identity), and Qustion~\ref{Q} was stated with this definition in mind. However, the answer to the original statement of Question~\ref{Q} was shown to be negative by Jacques and Short \cite[Section 16]{JS}. Question~\ref{Q} as we have stated it is a restatement of the questions of Yoccoz and Avila, Bochi and Yoccoz by Jacques and Short (see the last part of \cite[Section 16]{JS}). We also note that not all results from \cite{Yo2004} and \cite{AvBoYo2010} hold for the elliptic locus as defined in Definition~\ref{el}; for example our elliptic locus is no longer an open set. However, one could easily check that the results concerning $\EL$ mentioned above still remain valid.

In order to tackle Question~\ref{Q} we introduce a new parameter locus in $\mathrm{SL}(2,\mathbb{R})^N$. For any $\A=(A_1,A_2,\dots, A_N)\in\mathrm{SL}(2,\mathbb{R})^N$ we define the \emph{semigroup generated by $\A$} to be the semigroup of matrices generated by $\{A_1,A_2,\dots, A_N\}$ and we denote it by $\langle \A \rangle$.

\begin{definition}
We say that $\A\in\mathrm{SL}(2,\mathbb{R})^N$ is \emph{semidiscrete} if the matrices in $\langle \A \rangle$ do not accumulate at the identity. Furthermore, $\A$ is \emph{inverse-free} if $\langle \A \rangle$ does not contain the identity matrix.\\
The set $\SD$ of all semidiscrete and inverse-free $N$-tuples is called \emph{the semidiscrete and inverse-free locus}.
\end{definition}

The semidiscrete and inverse-free properties were introduced by Jacques and Short in \cite{JS}, where they were used in order to study semigroups of isometries of the hyperbolic plane. Observe that an $N$-tuple $\A$ is semidiscrete and inverse-free if and only if $I\notin\overline{\langle \A \rangle}$ where $I$ is the $2\times2$ identity matrix.

By Definition~\ref{hy} all uniformly hyperbolic $N$-tuples are semidiscrete and inverse-free, and so $\HY$ is contained in $\SD$. A great portion of this article is dedicated to studying the interaction between the loci $\HY$ and $\SD$. Our most significant contribution to this matter is the fact that points on the boundary of the hyperbolic locus, apart a few exceptional ones, are semidiscrete and inverse-free. In order to make this precise we require the following definition.

\begin{definition}\label{princ}
A connected component $H$ of the hyperbolic locus is called \emph{principal} if there exists a matrix $B\in\SL$, such that for any $(A_1,A_2,\dots,A_N)\in H$ the matrices $BA_1B^{-1}$, $BA_2B^{-1}$, $\dots$, $BA_NB^{-1}$ are positive.\\
The union of all principal components is denoted by $H_P$.
\end{definition} 

The precise statement of our result is the following.

\begin{theorem}\label{clweak}
The closure of $\HY\setminus H_P$ is contained in $\SD$.
\end{theorem}

In fact, we are going to prove that $\SD\setminus H_P$ is closed, which is a slightly stronger statement (see Theorem~\ref{cl}). Moreover, one could easily find examples of $N$-tuples on the boundary of a principal component that are not semidiscrete and inverse-free (see Section~\ref{sdlocus}). Thus the inclusion in the theorem cannon be improved significantly.

Theorem~\ref{clweak} indicates that principal components are special in the sense that they are the only components of $\HY$ whose boundary contains $N$-tuples that are not semidiscrete and inverse-free. All other boundary points of $\HY$---even buried points, if those exist---are semidiscrete and inverse-free. The proof of Theorem~\ref{clweak} and its stronger version rely heavily on the properties of the semigroups generated by semidiscrete and inverse-free $N$-tuples, and will be carried out in Section~\ref{sdlocus}. This section also contains certain results about the points on the boundary of components of $\HY$ outside of $H_P$.

Even though Theorem~\ref{clweak} does not appear to directly relate to Question~\ref{Q}, it plays an important role in proving that our construction in Section~\ref{counter} is indeed a counterexample to Question~\ref{Q}. In the final section of this article we present an updated version of Question~\ref{Q} that arises from the introduction of the semidiscrete and inverse-free locus, as well as two related questions regarding the structure of $\SD$.\\

\noindent \textbf{\large Acknowledgements.} The author  would like to thank Dr Ian Morris, whose suggestions helped improve the presentation of the initial draft. This work was financially supported by the \emph{Leverhulme Trust} (Research Project Grant number RPG-2016-194).

\section{Preliminaries}

In this section we present the main tools of our analysis.

\subsection{A quotient of parameter space}

Observe that if $\A=(A_1,A_2,\dots, A_N)$ is a uniformly hyperbolic $N$-tuple then, by Definition~\ref{hy}, any $N$-tuple of the form $(\pm A_1,\pm A_2,\dots,\pm A_N)$ is also uniformly hyperbolic. Thus, when studying the topological properties of the hyperbolic locus, it is handy to consider the parameter space to be $\PSL^N$ instead of $\SL^N$, where $\PSL$ is the quotient $\SL/\{\pm I\}$.

One can think of $\PSL^N$ as a copy of $\SL^N$ where all the components of $\HY$ that exhibit similar topological and dynamical properties are folded into one component. This makes discussing different components of $\HY$ less cumbersome and we continue to use $\PSL$ instead of $\SL$ for the rest of this article.

\subsection{Action on projective space and uniform hyperbolicity}

One of the main components of our work is the action of the group $\PSL$ on the projective real line $\R\P^1$. The model for $\R\P^1$ that we will be using throughout is the \emph{extended real line}, which is the real line together with the point infinity $\overline{\R}=\R\cup\{\infty\}$. Identifying $\R\P^1$ with $\overline{\R}$ allows us to describe the action of $\PSL$ on projective space through the following map:
\[
\PSL\ni A=\begin{pmatrix}
a & b\\
c & d
\end{pmatrix}
\mapsto
f_A(z)=\frac{az+b}{cz+d}.
\]
So $\PSL$ acts on $\R\P^1=\overline{\R}$ through real M\"obius transformations. In \cite[Theorem~2.2]{AvBoYo2010} the authors gave the following geometric characterisation of uniform hyperbolicity in terms of the projective action of $\PSL$.

\begin{theorem}\label{abyuh}
An $N$-tuple $\A=(A_1,A_2,\dots,A_N)\in\PSL^N$ is uniformly hyperbolic if and only if there exists a finite union $M$ of open intervals in $\overline{\R}$, with disjoint closures, such that each $f_{A_i}$, for $i=1,2,\dots,N$, maps $\overline{M}$ into $M$.
\end{theorem}

Theorem~\ref{abyuh} shows that the uniform exponential growth of the norms that defines uniform hyperbolicity (Definition~\ref{hy}) is equivalent to some form of uniform contraction for the iterated function system on $\R\P^1$ induced by $N$-tuples in $\PSL^N$. Furthermore, it allows us to shift our focus from matrices in $\PSL$ to real M\"obius transformations. That is, we can write $\PSL$ as follows
\[
\PSL=\left\{z\mapsto \frac{az+b}{cz+d}\colon a,b,c,d \in \mathbb{R} \quad \text{and}\quad ad-bc=1\right\}.
\]
Thus, staying consistent with Theorem~\ref{abyuh}, we say that an $N$-tuple $\F=(f_1,f_2,\dots,f_N)$ in $\PSL^N$ is \emph{uniformly hyperbolic} if there exists a finite union $M$ of open intervals in $\overline{\R}$, with disjoint closures, such that each $f_i$, for $i=1,2,\dots,N$, maps $\overline{M}$ into $M$. The set $M$ is called a \emph{multicone} of $\F$.

From now on, and for the rest of this article, we shall use $\F$ to denote an $N$-tuple in $\PSL^N$. We also imbue $\PSL$ with the topology of locally uniform convergence of holomorphic functions and $\PSL^N$ with the induced product topology. The main advantage of this identification for $\PSL$ is that it allows us to utilise the interaction between M\"obius transformations and hyperbolic geometry that we describe in the following section.

\subsection{M\"obius transformations and hyperbolic geometry}

Note that the group $\PSL$ also acts on the upper half-plane $\H=\{z\in\C\colon \Im{z}>0\}$ and so we can consider $\PSL$ acting on $\overline{\H}=\H\cup\overline{\R}$. If we equip the upper half-plane with the hyperbolic distance $\rho$, induced by the Riemannian metric $\frac{\lvert dz \rvert}{\Im{z}}$, then $\PSL$ is exactly the group of isometries of the metric space $(\H,\rho)$. For more information on the hyperbolic geometry of $\H$, we refer to \cite{Be1995,BeMi2007,Ka1992}.

The non-identity elements of $\PSL$ are classified as \emph{elliptic}, \emph{parabolic} or \emph{hyperbolic} depending on whether they have one fixed point in $\mathbb{H}$, one fixed point in $\overline{\mathbb{R}}$, or two fixed points in $\overline{\mathbb{R}}$, respectively.

Suppose that $f$ is a hyperbolic transformation in $\PSL$ and let $f^n\vcentcolon=f\circ f\circ \dots \circ f$ denote the $n$th iterate of $f$. The fixed point $\alpha(f)$ of $f$ such that the sequence $(f^n)$ converges to $\alpha(f)$ locally uniformly in $\H$, is called \emph{the attracting fixed point} of $f$. The other fixed point of $f$ is called the \emph{repelling fixed point}, and will be denoted by $\beta(f)$. Note that since $f$ is an isometry of the hyperbolic metric, it preserves the unique geodesic that joins $\alpha(f)$ and $\beta(f)$. This geodesic is called the \emph{axis} of $f$ and is denoted by $\text{Ax}(f)$.

If $f$ and $g$ are transformations in $\PSL$, then the \emph{commutator} of $f$ and $g$ is defined to be the transformation $[f,g]=f\circ g\circ f^{-1}\circ g^{-1}$. Also, for a transformation $h(z)=(az+b)/(cz+d)$ with $ad-bc=1$, we define the \emph{trace} of $h$ to be the number $\mathrm{tr}(h)= a+d$. The trace of a transformation provides us with the following classification of elements in $\PSL$.\\
Let $h$ be a transformation in $\PSL$. Then $h$ is:
\begin{enumerate}
\item parabolic or the identity if and only if $\lvert\mathrm{tr}(h)\rvert=2$;
\item elliptic if and only if $\lvert\mathrm{tr}(h)\rvert<2$; and
\item hyperbolic if and only if $\lvert\mathrm{tr}(h)\rvert>2$.
\end{enumerate}

In our figures, a hyperbolic transformation will be portrayed by drawing its axis as a directed hyperbolic line pointing towards its attracting fixed point (see Figure~\ref{psl} on the left). A parabolic transformation will be portrayed as a directed Euclidean circle, of arbitrary Euclidean radius, that is tangent to $\mathbb{R}$ at the fixed point of the parabolic transformation. The direction of the circle indicates the action of the transformation on $\mathbb{H}$ (see Figure~\ref{psl} on the right). In addition, it is often convenient to move our figures to the unit disc model of hyperbolic space, and we will do so without explicit mention of this fact since the half-plane and the disc models of hyperbolic space are conformally equivalent.

We end this section with some final pieces of notation. Suppose that $D$ is a subset of $\C$ and $f\colon D\to\C$ is conformal. We say that a set $A\subset D$ is \emph{forward invariant under} $f$ if $f(A)\subseteq A$. Similarly, $A$ is \emph{backward invariant under} $f$ if $f^{-1}(A)\subset A$. Also, $A$ is \emph{mapped compactly inside itself} by $f$ if $\overline{f(A)}\subset A$, while $A$ is \emph{mapped strictly inside itself} by $f$ if $f(A)\subset A$ with $f(A)\neq A$.

\begin{figure}[ht]
\centering
\begin{tikzpicture}
\begin{scope}
\draw (-1.6,0) -- (1.6,0);
\draw[directed] (1,0) arc (0:180:1)node[pos=0.45, above]{$f$};

\node[left] at (-1.6,0) {$\mathbb{R}$};
\node[below] at (-1,0) {$\alpha(f)$};
\node[below] at (1,0) {$\beta(f)$};
\end{scope}

\begin{scope}[xshift=5cm]

\draw (-1.6,0) -- (1.6,0);
\draw[directed] (0,0.5) circle (.5);
\draw[directed, dashed] (0,0.75) circle (.75);
\draw[directed, dashed] (0,1) circle (1);

\node[left] at (-1.6,0) {$\mathbb{R}$};
\node at (0,0.5) {$g$};
\end{scope}

\end{tikzpicture}
\caption{A hyperbolic transformation $f$ and a parabolic transformation $g$.}\label{psl}
\end{figure}
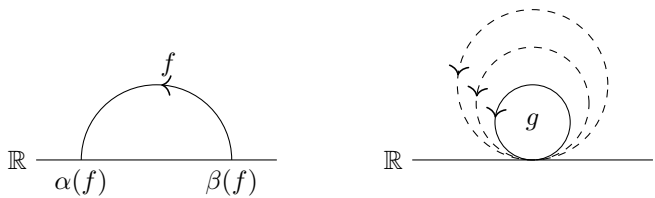

\section{Semigroups of M\"obius transformations}\label{sgsection}

As we mentioned in the introduction, our main techniques involve semigroups generated by a finite collection of real M\"obius transformations. In this section we review the results from this theory that are most relevant to our context.

Recall that we associate a semigroup to each $N$-tuple in $\PSL^N$ in the following way.

\begin{definition}
For $\F=(f_1,f_2,\dots,f_N)$ in $\text{PSL}(2,\mathbb{R})^N$, we define the \emph{semigroup generated by $\F$},  which we will denote by $\langle \F \rangle$, to be the semigroup generated by $\{f_1,f_2,\dots,f_N\}$. 
\end{definition}

Similarly, if $\{f_1,f_2,\dots f_n\}$ is any finite set of M\"obius transformations, we use $\langle f_1,f_2,\dots,f_n\rangle$ to denote the \emph{semigroup} that they generate. Semigroups of M\"obius transformations have been previously studied by Fried, Marotta and Stankiewitz \cite{FrMaSt2012}, as a particular branch of the theory of semigroups of rational functions initiated by Hinkkanen and Martin \cite{HiMa1996}. Our analysis is closely related to the work of Jacques and Short \cite{JS} who further developed the material in \cite{FrMaSt2012} by incorporating well-known results from the theory of Fuchsian groups.

A subset of $\PSL$ is called \emph{discrete} if the topology it inherits from $\PSL$ is the discrete topology. We now recall the definition of the semidiscrete and inverse-free locus from the introduction in light of our shift in focus discussed in the previous section. 

\begin{definition}
An $N$-tuple $\F\in\PSL^N$ is called \emph{semidiscrete} if the identity function is not an accumulation point of $\SF$. We say that $\F$ is \emph{inverse-free} if $\SF$ contains no inverses of its elements.
\end{definition}

We can easily see that $\F$ is semidiscrete and inverse-free if and only if $\mathrm{Id} \notin \overline{\SF}$, where $\mathrm{Id}$ is the identity transformation. Also, recall that $\SD$ is the parameter locus of all semidiscrete and inverse-free $N$-tuples in $\PSL^N$, called the semidiscrete and inverse-free locus.

The semidiscrete and inverse-free properties can be expressed geometrically using the following theorem from \cite[Theorem 7.1]{JS}.

\begin{theorem}[\cite{JS}]\label{JSmain}
An $N$-tuple $\F\in\PSL^N$ is semidiscrete and inverse-free if and only if there is a nontrivial closed subset $X$ of $\overline{\mathbb{H}}$ that is mapped strictly inside itself by each ordinate of $\F$.
\end{theorem}

Comparing Theorems \ref{abyuh} and \ref{JSmain} we can see a geometric interpretation of the fact that the hyperbolic locus $\HY$ is contained in the semidiscrete and inverse-free locus $\SD$. The two loci however are certainly not identical as the $N$-tuples in the following lemma, taken from \cite[Lemma 10.4]{JS}, demonstrate.

\begin{lemma}[\cite{JS}]\label{hump}
Suppose that $\F=(f,g)$, where $f(z)=az+b$ and $g(z)=cz+d$ are two hyperbolic transformations, with $0<c<1<a$ and $b/(1-a)<d/(1-c)$. The pair $\F$ is semidiscrete and inverse-free, and the closure of the semigroup $\SF$ contains all transformations of the form $z+t$, for $t\geq d/(1-c)-b/(1-a)$.
\end{lemma}

Observe that the pair $\F$ in Lemma \ref{hump} consists of two hyperbolic transformations with a common fixed point, which is attracting for one transformation and repelling for the other. Hence, $\F$ is not unifomrly hyperbolic, but $\F\in\SD$. Also, observe that this pair $\F$ illustrates that the discrete and semidiscrete properties are not equivalent for semigroups of M\"obius transformations, even though they are equivalent for groups (see, for example, \cite[page 14]{Be1995} or \cite[page 26]{Ka1992}).

We now introduce a class of semidiscrete and inverse-free $N$-tuples which exhibit ``simple" geometric properties, that will prove important to our analysis. 

\begin{definition}\label{sch}
Let $\F\in\SD$. If there exists a union $C$ of $m$ open intervals in $\overline{\mathbb{R}}$, with disjoint closures, such that each ordinate of $\mathcal{F}$ maps $C$ strictly inside itself, then $\F$ is said to be of \emph{finite rank}.
\end{definition}
Suppose that the integer $m$ in Definition~\ref{sch} is the smallest integer with this property, in the sense that if there exists another union $D$ of $n$ open intervals in $\overline{\mathbb{R}}$, with disjoint closures, such that each ordinate of $\F$ maps $D$ strictly inside itself, then $m\leq n$. We then say that $\F$ has \emph{rank $m$}.

By Theorem~\ref{JSmain} every $N$-tuple with finite rank is semidiscrete and inverse-free. Furthermore, every uniformly hyperbolic $\F$ has finite rank. We note that in the language of \cite{JS}, the semigroups generated by $N$-tuples of finite rank are called \emph{Schottky semigroups} due to their connection with the well-known class of discrete groups of the same name. It is important to note that not every semidiscrete and inverse-free $N$-tuple has finite rank as the example in \cite[Section 7]{JS} indicates.

In accordance to the theory of Fuchsian groups, we make the following definition.

\begin{definition}
An $N$-tuple $\F\in \PSL^N$ is called \emph{elementary} if the semigroup $\SF$ has a finite orbit in $\overline{\H}$.
\end{definition}

It is easy to check that $\F$ is elementary if and only either all ordinates of $\F$ have a common fixed point in $\overline{\H}$, or there exists a pair of points in $\overline{\R}$ that is fixed setwise by each ordinate. If $\F$ is not elementary we say that it is \emph{non-elementary}.

A key object in the theory of semigroups (or even groups) of M\"obius transformations is the notion of the limit set, which we now introduce.

\begin{definition}
Let $\F\in\PSL^N$. We define the \emph{forward limit set} $\LF(\F)$ of the semigroup $\SF$ to be the set accumulation points of $\{f(z_0) \colon f\in \SF\}$ in $\overline{\mathbb{R}}$, where $z_0\in\mathbb{H}$, with respect to the chordal metric in $\overline{\mathbb{C}}$. Similarly, the \emph{backward limit set} $\LB(\F)$ of $\SF$ is defined to be the forward limit set of $\SF^{-1}\vcentcolon=\{f^{-1}\colon f\in \SF\}$.
\end{definition}

It is a simple exercise to verify that the limit sets are independent of the choice of $z_0$, and so we can choose $z_0$ to be the complex number $i$. Moreover, $\LF(\F)$ is forward invariant under transformations in $\SF$ whereas $\LB(\F)$ is backward invariant. Fried, Marotta and Stankewitz \cite[Theorem 2.4, Proposition 2.6, and Remark 2.20]{FrMaSt2012} showed that whenever $\SF$ contains hyperbolic transformations, $\LB(\F)$ is the closure of all the repelling fixed points of hyperbolic  elements of $\SF$. Equivalently, if $\SF$ contains hyperbolic transformations, then $\LF(\F)$ is the closure of all the attracting fixed points of hyperbolic transformations in $\SF$. The version of this result that we will be using is stated in following, where the key idea for the proof is borrowed from \cite[Theorem 5.1.3]{Be1995}.

\begin{lemma}\label{nonel}
If $\F\in\PSL^N$ is non-elementary, then $\LF(\F)$ is the closure of all the attracting fixed points of hyperbolic transformations in $\SF$.
\end{lemma}

\begin{proof}
Due to the results of Fried, Marotta and Stankewitz we mentioned earlier, it suffices to show that if $\F$ is non-elementary then $\SF$ contains hyperbolic transformations. We first claim that if $\SF$ contains only elliptic transformations and $\mathrm{Id}$, then all ordinates of $\F$ have a common fixed point in $\H$.\\
Let us assume that $\SF$ contains only elliptic transformations and $\mathrm{Id}$. Suppose that $\SF$ contains an elliptic transformation $f$ of infinite order. If there exists another elliptic transformation $g$ that does not have the same fixed point as $f$, then $f$ and $g$ do not commute. Hence, \cite[Theorem~3]{BaBeCa1996} is applicable and implies that the semigroup generated by $f$ and $g$ is a dense subset of $\PSL$. Thus the semigroup $\langle f,g \rangle$, which is contained in $\SF$, contains non-elliptic transformations, a contradiction. We conclude that if $\SF$ contains an elliptic transformation of infinite order then all transformations in $\SF$ have a common fixed point in $\mathbb{H}$.\\
Suppose now that $\SF$ contains only elliptic transformations of finite order. Then $\SF$ contains all inverses of its elements, and thus is a purely elliptic group. The result in this case follows from \cite[Theorem~4.3.7]{Be1995}, which states that a subgroup $G$ of $\PSL$ is purely elliptic if and only if every element of $G$ fixes the same point in $\mathbb{H}$. This concludes the proof of our claim.\\
Assume, towards a contradiction, that $\SF$ does not contain any hyperbolic transformations. Our claim then implies that $\SF$ has to contain a parabolic transformation $p$. We can then conjugate $\SF$ by a M\"obius transformation so that $p(z)=z+1$. Let 
\[
f(z)=\frac{az+b}{cz+d},
\]
be a transformation in $\SF$. Then the transformation 
\[
p^n\circ f(z)=\frac{(a+nc)z+b+nd}{cz+d},
\]
has to satisfy $\lvert\tr(p^nf)\rvert\leq 2$, for all $n\in\N$. This implies that $c=0$ and so every transformation in $\SF$ fixes the point at infinity, which is a contradiction.
\end{proof}

As an example, we will evaluate the limit sets of the pair introduced in Lemma~\ref{hump}.

\begin{lemma}\label{limitset}
Suppose that $\F=(f,g)$, where $f(z)=az$ and $g(z)=cz+d$ with $0<c<1<a$ and $0<d/(1-c)$. The forward limit set of $\SF$ is the closed interval $[d/(1-c),\infty]$.
\end{lemma}

\begin{proof}
Note that the interval $[d/(1-c),\infty]$ is mapped into itself by $f$ and $g$ and so it has to contain all the attracting fixed points of hyperbolic elements of $\SF$. Thus $\LF(\F)\subseteq [d/(1-c),\infty]$. Also, observe that
\begin{equation}\label{gf}
g^m\circ f^n(z)=a^nc^mz+\frac{d}{1-c}(1-c^m),
\end{equation}
for any $m,n\in\mathbb{N}$. Suppose, first, that $a^nc^m\neq1$, for all positive integers $m,n$ and let $0<\lambda<1$. Then for every $\epsilon>0$, there exist $m_0$ and $n_0$ such that $|a^{n_0}c^{m_0}-\lambda|<\epsilon$ and $\lvert  \tfrac{d}{1-c}c^{m_0}\rvert<\epsilon$. Therefore, 
\[
\left|g^{m_0}\circ f^{n_0}(z)-\left(\lambda z+\frac{d}{1-c}\right)\right|<\epsilon(\lvert z\rvert +1),
\]
for all $z\in\mathbb{H}$. Observe that the finite fixed point of the transformation $\lambda z+d/(1-c)$ is $\tfrac{d}{(1-c)(1-\lambda)}$, and it is an attracting fixed point because $\lambda<1$. We deduce that in this case, $\tfrac{d}{(1-c)(1-\lambda)}$ lies in $\LF(\F)$, for all $0<\lambda<1$, which yields the desired result.\\
If on the other hand, $a^{\mu}b^{\nu}=1$ for some $\mu,\nu\in\mathbb{N}$ then, by considering the subsemigroup $\langle f^{\nu}, g^{\mu} \rangle$ of $\SF$, we can assume that $ac=1$. So, from \eqref{gf} 
\[
H_n(z)=g^n\circ f^n(z)=z+\frac{d}{1-c}(1-c^n),
\]
for all $n\in\mathbb{N}$. So for all non-negative integers $n,l$, the transformation
\[
g^{2n}\circ {H_n}^l\circ f^n(z)=c^nz+l\frac{d}{1-c}(1-c^n)c^{2n}+\frac{d}{1-c}(1-c^{2n}),
\]
is a hyperbolic transformation in $\SF$, whose attracting fixed point is
\[
\frac{d}{1-c}(lc^{2n}+c^n+1).
\]
For every real number $x>0$, and every $\epsilon>0$, we can choose positive integers $l,n$, so that $\lvert lc^{2n}-x\rvert<\tfrac{\epsilon}{2}$ and $c^{n}<\tfrac{\epsilon}{2}$. This implies that for every $r>1$, and all $\epsilon>0$, we can find $l,n$ such that $\lvert lc^{2n}+c^n+1- r\rvert<\epsilon$. Hence, the point $r\, d/(1-c)$ lies in $\LF(\F)$ for all $r>1$.
\end{proof}

Observe that $\F$ is inverse-free if and only if $\SF\cap \SF^{-1}=\emptyset$. It is also interesting to note that $\SF \cap \SF^{-1}$, if non-empty, is a group whose limit set lies in the intersection $\LF(\F)\cap\LB(\F)$. So, one might expect that there is a connection between the size of the intersection of the limit sets and the size of $\SF\cap \SF^{-1}$. Obviously, if $\SF$ is a group then the backward and forward limits sets coincide. For non-elementary $N$-tuples the following theorem \cite[Theorem 1.9]{JS} shows that the other direction of this statement is also true.

\begin{theorem}[\cite{JS}]\label{JSgroup}
Let $\F$ be a non-elementary and semidiscrete $N$-tuple. If $\LB(\F)\subseteq \LF(\F)$, then $\SF$ is a group.
\end{theorem}

Using Theorem~\ref{JSgroup} we establish the the following lemma, where $\LF(\F)^{\mathrm{o}}$ denotes the interior of the forward limit set.

\begin{lemma}\label{nonsd}
Let $\F$ be a non-elementary $N$-tuple such that $\SF$ is not a discrete group. If $\LF(\F)^{\mathrm{o}}\cap\LB(\F)\neq \emptyset$ then $\F$ is not semidiscrete. 
\end{lemma}

\begin{proof}
Since $\F$ is non-elementary, Lemma~\ref{nonel} implies that $\LB(\F)$ is the smallest closed set containing all the repelling fixed points of hyperbolic elements of $\SF$. Therefore, there exists a hyperbolic transformation $f$ in $\SF$ with repelling fixed point in $\LF(\F)^o$. So, by the invariance of the limit sets under the semigroup, we have that $\overline{\mathbb{R}}=\overline{\{f^n(\Lambda^+(\F)^{\mathrm{o}})\colon n\in\mathbb{N}\}})\subseteq \LF(\F)$ and therefore $\LF+(\F)=\overline{\mathbb{R}}$. Finally, because $\SF$ is non-elementary and $\LB(\F)\subset\overline{\mathbb{R}}=\LF(\F)$, Theorem~\ref{JSgroup} tells us that if $\SF$ were semidiscrete, it would have to be a discrete group, which is a contradiction.
\end{proof}

We end this section with a lemma about the limit set of semigroups generated by pairs in $\PSL^2$.

\begin{lemma}\label{ls-inter}
Let $\F=(f,g)$ be an $N$-tuple, where $f$ and $g$ are transformations in $\PSL$, with $f(x)=x$ and $g(y)=y$, for some points $x<y$ in $\overline{\mathbb{R}}$. Also, assume that the open interval $(x,y)$ is mapped strictly inside itself by $f$ and $g$. Then the forward limit set $\LF(\F)$ is the closed interval $[x,y]$ if and only if $g(x)\leq f(y)$.
\end{lemma}

\begin{proof}
Note that since $f$ and $g$ map $(x,y)$ strictly inside itself, they are either parabolic or hyperbolic with attracting fixed points in $\{x,y\}$. We define $I = [x,y]$, and note that $I$ is invariant under $\langle f,g\rangle$, which implies that $\LF(\F)$ is contained in $I$. If $f(y)<g(x)$, then the intervals $f(I)$ and $g(I)$ are disjoint and $\LF(\F)\subset f(I)\cup g(I)$, which implies that $\LF(\F)$ is a proper subset of $I$.\\
For the converse, assume that $g(x)\leq f(y)$. Then $g(I)\cup f(I)= I$, and thus for every $c\in I$ there exists $f_1\in\{f,g\}$ such that $c\in f_1(I)$. So we can recursively find a sequence $(f_n)$ with $f_n\in\{f,g\}$, and such that $c\in f_1\circ f_2\circ \dots \circ f_n(I)$, for all $n=1,2,\dots$. It is easy to check that the intervals $f_1\circ f_2\circ \dots \circ f_n(I)$ are nested and their Euclidean length converges to 0 as $n\to\infty$. Hence, every $c\in I$ is an accumulation point of either $x$ or $y$ under $\langle f,g\rangle$, which implies that $\LF(\F)=I$.
\end{proof}

\section{The semidiscrete and inverse-free locus}\label{sdlocus}

We now investigate the connection between the loci $\HY$ and $\SD$. We start with a characterisation of uniformly hyperbolic $N$-tuples in terms of the limits sets of the semigroups they generate. This result can be easily inferred from the material in \cite[Section 2.4]{AvBoYo2010}; we provide the proof for the sake of completeness.

\begin{lemma}[\cite{AvBoYo2010}]\label{lsuh}
An $N$-tuple $\F$ is uniformly hyperbolic if and only if all ordinates of $\F$ are hyperbolic transformations and $\LF(\F)\cap\LB(\F)=\emptyset$.
\end{lemma}

\begin{proof}
Suppose that $\F$ is uniformly hyperbolic and let $M$ be its multicone. Recall that uniform hyperbolicity implies that $\SF$ only contains hyperbolic transformations. Since $M$ is mapped compactly inside itself, all the attracting fixed points of elements of $\SF$ lie in $M$. Also, $M$ does not contain any repelling fixed points of elements of $\SF$. Hence, as $\LF(\F)$ and $\LB(\F)$ are the closures of the attracting and repelling fixed points of elements of $\SF$, respectively, $\LF(\F)$ is contained in $M$ and does not intersect $\LB(\F)$.\\
For the converse, note that because all the ordinates of $\F$ are hyperbolic and the limit sets are disjoint, $\SF$ has to be inverse-free. Using arguments similar to the proof of Lemma~\ref{corelemma}, it is easy to check that $\LF(\F)\cap\LB(\F)=\emptyset$ implies that $C^+(\F)\cap C^-(\F)=\emptyset$ and the sets $C^+(\F)$ and $C^-(\F)$ are forward and backward invariant, respectively. Thus \cite[Lemma 2.7]{AvBoYo2010} is applicable and yields that $\F$ is uniformly hyperbolic.
\end{proof}

For simplicity, we call the connected components of $\HY$ \emph{hyperbolic components}. We now present an alternative definition for the principal hyperbolic components, which is equivalent to Definition~\ref{princ} (their equivalence be easily inferred from \cite[Proposition 3]{Yo2004}).

\begin{definition}
The component $H_P$ of $\HY$ that contains all uniformly hyperbolic $N$-tuples that have a multicone consisting of a single interval is called \emph{the principal component} of $\HY$.
\end{definition}

Note that since we are working in $\PSL^N$ instead of $\SL^N$ there is only one principal component. Closely related to the principal component is the following set. 

\begin{definition}\label{mjd}
Let $J$ be an open, nontrivial interval in $\overline{\mathbb{R}}$. We define $\mathcal{M}(J)$ to be the semigroup of $\PSL$ that map $J$ inside itself.
\end{definition}

The following lemma describes the connection between $H_P$ and $\mathcal{M}(J)$ that was hinted at earlier.

\begin{lemma}\label{mj}
Suppose that $\F\in\PSL^N$ is such that $\SF$ lies in $\mathcal{M}(J)$, for some open, nontrivial interval $J\subset\overline{\mathbb{R}}$. Then $\F$ lies in the closure of the principal component.
\end{lemma}

\begin{proof}
Assume, without loss of generality, that $J=(0,1)$. Observe that the transformations in $\mathcal{M}(J)$ are either hyperbolic transformations whose attracting fixed points lie in $\overline{J}$ and repelling fixed points in $J^c$, parabolic transformations that fix one endpoint of $J$ and map $J$ inside itself, or the identity. Now, let $U$ be a neighbourhood of $\F$ in $\PSL^N$. Observe that a parabolic transformation that fixes 0 or 1 can be approximated by a sequence of hyperbolic transformations $(f_n)$, with $\alpha(f_n)\in (0,1)$ and $\beta(f_n)\in[0,1]^c$. Therefore, we can find $\F'$ in $U$ such that all ordinates of $\F'$ are hyperbolic transformations with attracting fixed points in $J$, and repelling fixed points outside $\overline{J}$. Hence, $\langle \F' \rangle$ maps $J$ compactly inside itself and thus $\F$ lies in $H_P$.
\end{proof}

Note that if $\F$ has rank one then $\SF$ is contained in $\mathcal{M}(J)$, for some interval $J$. So Lemma~\ref{mj} implies that $N$-tuples of rank one lie in the closure of $H_P$. It is also easy to see that all points in $H_P$ have rank one.  The converse, however, is not true since semigroups in $\mathcal{M}(J)$ are not necessarily generated by semidiscrete and inverse-free $N$-tuples. As an easy example, suppose that $(f,g)$ has rank one. Then the semigroup $\langle f,g,\mathrm{Id}\rangle$ lies in $\mathcal{M}(J)$, for some $J$, but $(f,g,\mathrm{Id})$ is not inverse-free. For a more substantial example, consider two hyperbolic transformations $f$ and $g$ with $\beta(f)=\alpha(g)$ and $\alpha(f)=\beta(g)$, and let $\F=(f,g)$. It is easy to check that $\F$ is not semidiscrete and inverse-free but lies on the boundary of $H_P$, as it can be approximated by a pair consisting of hyperbolic transformations with crossing axes. Therefore, the semidiscrete and inverse-free locus is not a closed set. However, we can prove the following.

\begin{theorem}\label{cl}
The set $\SD\setminus \overline{H_P}$ is closed in $\PSL^N$.
\end{theorem}

Note that Theorem~\ref{cl} implies Theorem~\ref{clweak} in the introduction. It also yields the following useful corollary.

\begin{corollary}\label{sc}
The closure of $\SD$ in $\PSL^N$ is $\SD\cup \partial H_P$.
\end{corollary}

In order to prove Theorem~\ref{cl} we are going to need certain preliminary results. First, we have a version of J{\o}rgensen's inequality for semigroups, which can be found in \cite[Theorem 12.11]{JS}.

\begin{theorem}[\cite{JS}]\label{jo}
Suppose that $f$ and $g$ are M\"obius transformations in $\PSL$ such that $(f,g)$  is non-elementary and semidiscrete. Then either  $(f,g )$ has rank one, or else
\[
\lvert \lvert\mathrm{tr}(f)\rvert^2 -4\rvert  + \lvert \lvert \mathrm{tr}([f,g])\rvert -2 \rvert\geq 1,
\]
where $[f,g]=f\circ g \circ f^{-1}\circ g^{-1}$ is the commutator of $f$ and $g$.
\end{theorem}

We recall the following definition from \cite[Section 12]{JS} (see also \cite[Lemma 12.2]{JS}).

\begin{definition}
Suppose that $f$ and $g$ are either hyperbolic or parabolic transformations in $\PSL$, with no common fixed points. The pair $f$ and $g$ is called \emph{antiparallel} if $(f,g)$ does \emph{not} have rank one.
\end{definition}

Observe that if $f$ and $g$ are antiparallel then $(f,g)$ does not lie in the closure of the principal component in $\PSL^2$. See Figure~\ref{antipara} for configurations of antiparallel pairs.

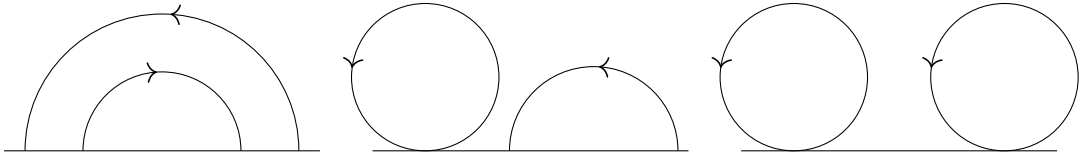
\begin{figure}[ht]
\centering
\begin{tikzpicture}[scale = 1.4]
\begin{scope}

\draw(-1.5,0) -- (1.5,0);
\draw[directed] (1.3,0) arc (0:180:1.3);
\draw[directed] (-.75,0) arc (180:0:0.75);
\end{scope}

\begin{scope}[xshift= 3.5cm]

\draw(-1.5,0) -- (1.5,0);
\draw[directed] (-1,0.7) circle (0.7);
\draw[directed] (1.4,0) arc (0:180:0.8);
\end{scope}

\begin{scope}[xshift= 7cm]

\draw(-1.5,0) -- (1.5,0);
\draw[directed] (-1,0.7) circle (0.7);
\draw[directed] (1,0.7) circle (0.7);
\end{scope}

\end{tikzpicture}
\caption{Pairs of antiparallel transformations.}\label{antipara}
\end{figure}

We first prove an elementary lemma about the commutator of M\"obius transformations.

\begin{lemma}\label{comm}
Let $f$ and $g$ be hyperbolic or parabolic transformations in $\PSL$ with a common fixed point. Then, for any transformations $h$ and $k$ in $\langle f,g \rangle$ we have that $\lvert \mathrm{tr}([h,k]) \rvert=2$.
\end{lemma}

\begin{proof}
Conjugate $f$ and $g$ by a transformation in $\PSL$ so that they both fix the point at infinity. Then, any transformation $\phi$ in $\langle f,g \rangle$ is of the form $\phi(z)=\lambda z+ \kappa$, for some $\lambda>0$ and $\kappa\in\mathbb{R}$. Suppose that $h(z)=az+b$ and $k(z)=cz+d$ lie in $\langle f,g \rangle$. Then we can easily see that the commutator of $h$ and $k$ is given by $[h,k](z)=z +(a-1)d - (c-1)d$. Therefore, $[h,k]$ is either parabolic or the identity, which yields the desired result.
\end{proof}

\begin{lemma}\label{pr1}
Suppose that $(f_n)$ and $(g_n)$ are sequences of hyperbolic or parabolic transformations in $\PSL$ such that $f_n$ and $g_n$ do not have any common fixed points and are antiparallel, for all positive integers $n$. Assume that $(f_n)$ converges to $f$ and $(g_n)$ converges to $g$, where $f,g$ are non-identity transformations that have a common fixed point in $\overline{\mathbb{R}}$. Then for all $n$ large enough, $(f_n,g_n)$ is not semidiscrete and inverse-free.
\end{lemma}

\begin{proof}
Let $\F_n=(f_n,g_n)$, $\F=(f,g)$, $S_n=\langle \F_n \rangle$ and $S=\langle \F \rangle$. Firstly, note that due to Lemma~\ref{comm}, we have that $\lvert \mathrm{tr}([h,k])\rvert=2$ for all transformations $h,k$ in $S$. Also, the commutator is a continuous function from $\PSL^2$ to $\PSL$, which implies that for any $h_n, k_n\in S_n$ with $h_n\to h$ and $k_n\to k$, the sequence $\lvert \mathrm{tr}([h_n,k_n]) \rvert$ converges to 2. We are going to use this fact throughout the proof.\\
Let us assume that $\F_n$ is semidiscrete and inverse-free, for infinitely many $n$. By relabelling the sequence $\F_n$, we can assume that $\F_n$ is semidiscrete and inverse-free for all $n$. Since the transformations $f,g$ have a common fixed point in $\overline{\mathbb{R}}$, and none of them is the identity, each of the $f,g$ is either hyperbolic or parabolic. Suppose that $f$ is parabolic. Note that the trace of a M\"obius transformation is continuous in $\PSL$, and so $\mathrm{tr}(f_n)$ converges to 2 as $n\to\infty$. So, since the trace of the commutator $[f_n,g_n]$ also converges to 2, we have that the quantity
\[
\lvert \lvert\mathrm{tr}(f_n)\rvert^2 -4\rvert  + \lvert \lvert \mathrm{tr}([f_n,g_n])\rvert -2 \rvert,
\]
is less than one, for all $n$ large enough. Hence, J{\o}rgensen's inequality for semigroups, Theorem~\ref{jo}, implies that in order for $\F_n$ to be semidiscrete, it would have to be of rank one. This is a contradiction because $f_n$ and $g_n$ are antiparallel, and so $\F_n$ cannot be semidiscrete. We reach the same conclusion if $g$ is parabolic.\\
So for the rest of the proof we can assume that $f$ and $g$ are hyperbolic. Note that the set of hyperbolic transformations of $\PSL$ is open. So because $f_n$ and $g_n$ converge to $f$ and $g$, respectively, they have to be hyperbolic transformations, for all $n$ large enough. Thus, by relabelling the sequences $(f_n)$ and $(g_n)$, we can assume that they are sequences of hyperbolic transformations. Also, since $f_n$ and $g_n$ are antiparallel, we have that the attracting fixed point of one of the $f$ and $g$ is the repelling fixed point of the other. Without loss of generality, we assume that $\beta(f)=\alpha(g)$, and conjugate by a M\"obius transformation so that $\alpha(f)<\beta(f)=\alpha(g)<\beta(g)$ (see Figure~\ref{789} on the right).\\
Due to our conjugation, for all $n$ large enough we have that $\alpha(f_n)<\beta(f_n)<\alpha(g_n)<\beta(g_n)$ (see Figure~\ref{789} on the left). Because $\beta(f)=\alpha(g)$, Lemma~\ref{hump} is applicable to $\F$ and implies that the closure of the semigroup $S$ in $\PSL$ contains parabolic transformations that fix $\beta(f)=\alpha(g)$ and map the interval $(\alpha(f),\alpha(g))$ strictly inside itself. So, there exists a sequence $(\phi_n)$ in $S$ that converges to such a parabolic transformation. By the convergence of $(f_n)$ and $(g_n)$ to $f$ and $g$, respectively, we can pass to a subsequence $(\F_{n_k})$ of $(\F_n)$ in order to obtain a sequence of transformations $(h_k)$ in $S_{n_k}$ such that $h_k$ converges to a parabolic transformation fixing $\beta(f)=\alpha(g)$ and mapping $(\alpha(f),\alpha(g))$ strictly inside itself. For convenience we relabel the sequences $(h_k)$ and $(\F_{n_k})$ so that $h_n$ lies in $S_n=\langle\F_n\rangle$, for all $n$.

\begin{figure}[ht]
\centering
\begin{tikzpicture}[scale = 1.4]
\begin{scope}

\draw(-2.15,0) -- (2.15,0);
\draw[directed] (-.2,0) arc (0:180:.9)node[pos=.45,above,yshift =3pt]{$f_n$};
\draw[directed] (2,0) arc (0:180:.9)node[pos=.45,above,yshift =3pt]{$g_n$};
\node[below] at (-2,0) {$\alpha(f_n)$};
\node[below, xshift=-.4cm] at (-.2,0) {$\beta(f_n)$};
\node[below, xshift=.4cm] at (.2,0) {$\alpha(g_n)$};
\node[below] at (2,0) {$\beta(g_n)$};
\end{scope}

\begin{scope}[xshift= 3cm]
\draw[->] (-.5,.7) -- (.5,.7)node[pos=0.5,above]{$n\to\infty$};

\end{scope}

\begin{scope}[xshift= 6cm]

\draw(-2.15,0) -- (2.15,0);
\draw[directed] (0,0) arc (0:180:1)node[pos=.45,above,yshift =3pt]{$f$};
\draw[directed] (2,0) arc (0:180:1)node[pos=.45,above,yshift =3pt]{$g$};
\node[below] at (-2,0) {$\alpha(f)$};
\node[below] at (0,0) {$\beta(f)=\alpha(g)$};
\node[below] at (2,0) {$\beta(g)$};
\end{scope}

\end{tikzpicture}
\caption{The configuration of $f$ and $g$.}\label{789}
\end{figure}
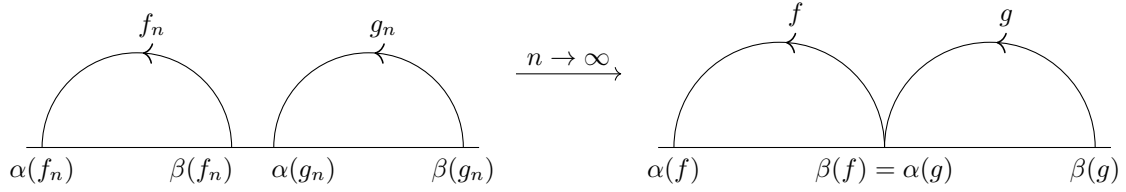

Note that because $\F_n$ is semidiscrete and inverse-free, $h_n$ has to be either hyperbolic or parabolic, for all $n$. Suppose that $h_n$ and one of $f_n$ or $g_n$ are antiparallel for infinitely many $n$. By passing to a subsequence if necessary and without loss of generality, we can assume that $h_n$ and $f_n$ are antiparallel, for all $n\in\mathbb{N}$ (see Figure~\ref{456}, $h_n$ is hyperbolic on the left-hand side and parabolic on the right). Then J{\o}rgensen's inequality for semigroups, Theorem~\ref{jo}, implies that 
\begin{equation}\label{joapp}
\lvert \lvert\mathrm{tr}(h_n)\rvert^2 -4\rvert  + \lvert \lvert \mathrm{tr}([f_n,h_n])\rvert -2 \rvert\geq 1,
\end{equation}
for all positive integers $n$. But the left-hand side of \eqref{joapp} converges to zero, as $h_n$ converges to a parabolic transformation and $\mathrm{tr}([f_n,h_n])$ converges to 2, which is a contradiction. We reach the same contradiction by assuming that $h_n$ and $g_n$ are antiparallel, for all $n\in\mathbb{N}$.

\begin{figure}[ht]
\centering
\begin{tikzpicture}[scale = 1.4]
\begin{scope}

\draw(-2.3,0) -- (2.3,0);
\draw[directed] (-.35,0) arc (0:180:.9)node[pos=.45,above,yshift =3pt]{$f_n$};
\draw[directed] (2.15,0) arc (0:180:.9)node[pos=.45,above,yshift =3pt]{$g_n$};
\draw[directed] (1.5,0) arc (0:180:.5)node[pos=.45,above,yshift =.5pt, xshift=3pt]{$h_n$};
\end{scope}

\begin{scope}[xshift= 6cm]

\draw(-2.3,0) -- (2.3,0);
\draw[directed] (-.35,0) arc (0:180:.9)node[pos=.45,above,yshift =3pt]{$f_n$};
\draw[directed] (2.15,0) arc (0:180:.9)node[pos=.45,above,yshift =3pt]{$g_n$};
\draw[directed] (0,0.3) circle (.3);
\node[above] at (0,0.6) {$h_n$};
\end{scope}
\end{tikzpicture}
\caption{The transformations $h_n$ and $f_n$ are antiparallel.}\label{456}
\end{figure}
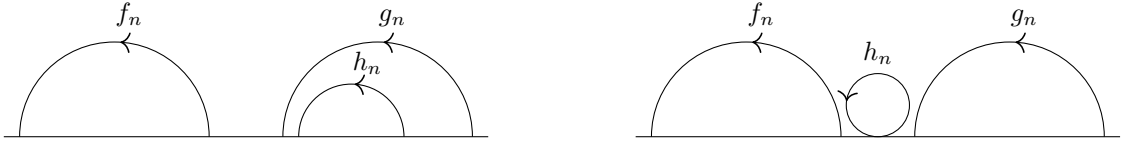

Finally, the only case left to consider is when none of the pairs $h_n, f_n$ and $h_n,g_n$ are antiparallel, for any $n$ large enough. Observe that if $h_n$ is parabolic, then one of the pairs $h_n, f_n$ and $h_n,g_n$ has to be antiparallel. Hence, all the transformations $h_n$ have to be hyperbolic, for all $n$ large enough. Also, due to the fact that the sequences of points $\beta(f_n),\alpha(h_n),\beta(h_n)$ and $\alpha(g_n)$ all converge to the same point $\beta(f)=\alpha(g)$, the axis of $h_n$ has to intersect both $\mathrm{Ax}(f_n)$ and $\mathrm{Ax}(g_n)$, for all $n$ large enough (see Figure~\ref{222}). As the sequence $h_n$ converges to a parabolic transformation that fixes $\beta(f)=\alpha(g)$ and maps $(\alpha(f),\beta(f))$ strictly inside itself, we must have that $\alpha(h_n)<\beta(f_n)$, for all $n$ large enough. Thus, we have the following order for the fixed points of the hyperbolic transformations $f_n, h_n$ and $g_n$
\begin{equation}\label{order}
\alpha(f_n)<\alpha(h_n)<\beta(f_n)<\alpha(g_n)<\beta(h_n)<\beta(g_n), \quad\text{for all } n\text{ large enough}.
\end{equation}

Note that the sequence of points $(h_n(\alpha(f_n)))$ converges to some point in the open interval $(\alpha(f),\beta(f))$, and the sequence $(f_n(\alpha(h_n)))$ converges to $\beta(f)$. Hence, $h_n(\alpha(f_n))\leq f_n(\alpha(h_n))$, for all $n$ large enough. We can then apply Lemma~\ref{ls-inter} to the transformations $f_n, h_n$ in order to obtain that $\LF(f_n, h_n)=[\alpha(f_n),\alpha(h_n)]$. Furthermore, by the convergence of $(f_n),(g_n)$ and $(h_n)$, the point $\beta(f_n)$ is contained in the interior of the interval $g_n([\alpha(f_n),\alpha(h_n)])\subset\LF(S_n)$, for all $n$ large enough. Therefore, the points $(\beta(f_n))$ lie in the interior of $\LF(\F_n)$ for all $n$ large enough. Lemma~\ref{nonsd} is then applicable to $\F_n$ and implies that the semigroup $S_n$ is either a discrete group or $\F_n$ is not semidiscrete. But, $\F_n$ is inverse-free, so it $S_n$ cannot be a group and we have reached a contradiction.\\
 All our contradictions were reached because we assumed that there exists a subsequence of $\F_n$ that is semidiscrete and inverse-free. Therefore, we conclude that only finitely many of $\F_n$ are semidiscrete and inverse-free.
\end{proof}

\begin{figure}[ht]
\centering
\begin{tikzpicture}[scale = 1.4]
\begin{scope}

\draw(-2.3,0) -- (2.3,0);
\draw[directed] (-.35,0) arc (0:180:.9)node[pos=.45,above,yshift =3pt]{$f_n$};
\draw[directed] (2.15,0) arc (0:180:.9)node[pos=.45,above,yshift =3pt]{$g_n$};
\draw[directed] (.7,0) arc (0:180:.7)node[pos=.45,above,yshift =3pt]{$h_n$};
\end{scope}
\end{tikzpicture}
\caption{The sequence of transformations $(h_n)$.}\label{222}
\end{figure}
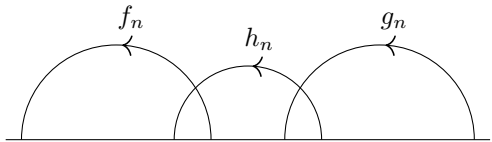

\begin{theorem}\label{pr}
Let $(\F_m)$ be a sequence in $\SD$ that converges to a point $\F\in\SD$. If $\F$ has rank one, then for all $m$ large enough $\F_m$ also has rank one. 
\end{theorem}

\begin{proof}
Write $\F=(f_1,f_2,\dots,f_N)$ and $\F_m=(f_1^m,f_2^m,\dots,f_N^m)$. Also, without loss of generality, suppose that $(0,1)\subset \overline{\mathbb{R}}$ is mapped strictly inside itself by $\SF$. Since $\F$ is semidiscrete and inverse-free and maps an interval strictly inside itself, all ordinates of $\F$ have to either be hyperbolic transformations with attracting fixed points in $[0,1]$ and repelling fixed points in the complement of $(0,1)$, or parabolic transformations that fix $0$ or $1$ and map $(0,1)$ inside itself. Let us assume that there exists a subsequence of $(\F_m)$, such that $\F_{m_k}$ is semidiscrete and inverse-free but not of rank one, for all $k$. By relabelling the sequence $(\F_{m_k})$, we can assume that $\F_m$ is semidiscrete and inverse-free but not of rank one, for all $m$.\\
We are first claim that $\F_m$ has to contain pairs of ordinates that are antiparallel, for all $m$ large enough. Let us assume otherwise. Then in order for $\F_m$ to not have rank one, for all large enough $m$, there must exists ordinates $f_i^m,f_j^m$ and $f_l^m$ of $\F_m$ with the following properties, up to conjugation: $f_i^m$ and $f_j^m$ are hyperbolic transformations, with crossing axes and such that $\alpha(f_i^m)<\alpha(f_j^m)<\beta(f_i^m)<\beta(f_j^m)$; the axis of $f_l^m$ intersects $\mathrm{Ax}(f_i^m)$ and $\mathrm{Ax}(f_j^m)$, and the repelling fixed point of $f_l^m$ (or its unique fixed point if it is parabolic) lies in the interval $(\alpha(f_i^m),\alpha(f_j^m))$ (see Figure~\ref{0909}). Also, by passing to a further subsequence if necessary, we can assume that $f_i^m,f_j^m$ and $f_l^m$  converge to some ordinates $f_i,f_j$ and $f_l$ of $\F$, respectively. Note that the attracting fixed points of $f_i^m,f_j^m$ and $f_l^m$ converge to some points in $[0,1]$. In addition, the transformations $f_i,f_j$ and $f_l$ have to map the interval $(0,1)$ strictly inside itself. Hence, the points $\alpha(f_i^m),\beta(f_l^m)$ and $\alpha(f_j^m)$ all converge to some endpoint of $(0,1)$, say 0 (recall that we assumed that no ordinates of $\F_m$ are antiparallel). In addition, because the axis of $f_l^m$ intersects $\mathrm{Ax}(f_i^m)$ and $\alpha(f_i^m)\leq \beta(f_l^m)$, we have that $\beta(f_i^m)\leq \alpha(f_l^m)$ for all $m$. But $\alpha(f_l^m)$ converges to some point in $[0,1]$, and so the transformations $f_i$ and $f_l$ have the same axes and $\alpha(f_i)=\beta(f_l)$. This implies that $\F$ is not semidiscrete and inverse-free and we have reached a contradiction.

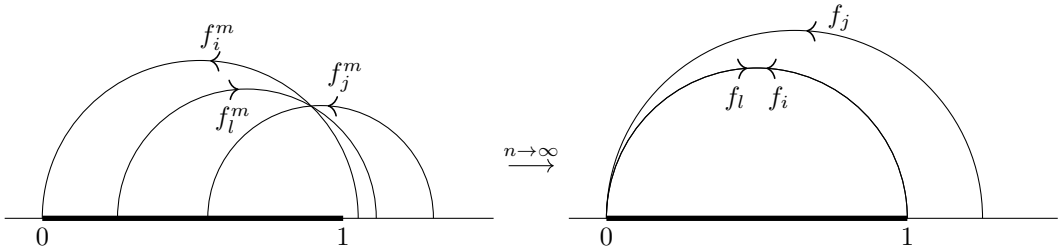
\begin{figure}[ht]
\centering
\begin{tikzpicture}

\begin{scope}
\draw(-2.5,0) -- (4,0);
\draw[directed] (2.2,0) arc (0:180:2.1)node[pos=0.465,above, yshift=1pt]{$f_i^m$};
\draw[directed] (3.2,0) arc (0:180:1.5)node[pos=0.43,above, yshift=3pt]{$f_j^m$};
\draw[directed] (-1,0) arc (180:0:1.72)node[pos=0.465,below, yshift=-3pt]{$f_l^m$};

\node[below] at (-2,0) {0}; 
\node[below] at (2,0) {1};
\draw[line width=2pt] (-2,0) -- (2,0);
\end{scope}

\begin{scope}[xshift= 4.5cm]
\draw[->] (-.3,.7) -- (.3,.7)node[pos=0.5,above]{$\scriptstyle n\to\infty$};
\end{scope}

\begin{scope}[xshift=7.5cm]
\draw(-2.5,0) -- (4,0);
\draw[directed] (3,0) arc (0:180:2.5)node[pos=0.4,above, xshift=-5pt]{$f_j$};
\draw[directed] (2,0) arc (0:180:2)node[pos=0.465,below, yshift=-3pt,xshift=2pt]{$f_i$};
\draw[directed] (-2,0) arc (180:0:2)node[pos=0.465,below, yshift=-3pt, xshift=-2pt]{$f_l$};

\node[below] at (-2,0) {0}; 
\node[below] at (2,0) {1};
\draw[line width=2pt] (-2,0) -- (2,0);
\end{scope}

\end{tikzpicture}
\caption{$\langle \F_m \rangle$ does not contain any antiparallel pairs.}\label{0909}
\end{figure}
So the fact that $\F_m$ does not have rank one implies that we can find ordinates $f_i^m$ and $f_j^m$ of $\F_m$ that are antiparallel, for all $m$ large enough. Without loss of generality, assume that $f_1^m$ and $f_2^m$ are antiparallel for all $m$ (see Figure~\ref{mnb}). By passing to a subsequence if necessary, we can assume that $f_1^m$ and $f_2^m$  converge to the ordinates $f_1$ and $f_2$ of $\F$, respectively. Because the attracting fixed points (or the unique fixed points if they are parabolic) of $f_1$ and $f_2$ have to lie in $[0,1]$, and the pair $f_1,f_2$ is approximated by a pair of antiparallel transformations, we have that $f_1$ and $f_2$ both fix an endpoint of $(0,1)$, say $1$. Then Lemma~\ref{pr1} is applicable to the sequences $(f_1^m)$ and $(f_2^m)$, and yields that $\langle \F_m \rangle$ is not semidiscrete and inverse-free, for any $m$ large enough, which is a contradiction.\\
We conclude that for all $n$ large enough, $\F_n$ has rank one.
\end{proof}

\begin{figure}[ht]
\centering
\begin{tikzpicture}

\begin{scope}
\draw(-2.5,0) -- (4,0);
\draw[directed] (1.7,0) arc (0:180:1.7)node[pos=0.465,above, yshift=1pt]{$f_1^m$};
\draw[directed] (3.5,0) arc (0:180:0.8)node[pos=0.45,above, yshift=1pt]{$f_2^m$};

\node[below] at (-2,0) {0}; 
\node[below] at (2,0) {1};
\draw[line width=2pt] (-2,0) -- (2,0);
\end{scope}

\begin{scope}[xshift= 7cm]
\draw(-2.5,0) -- (4,0);
\draw[directed] (3,0) arc (0:180:2.3)node[pos=0.465,above, yshift=1pt]{$f_1^m$};
\draw[directed] (-1,0) arc (180:0:1)node[pos=0.45,above, yshift=1pt]{$f_2^m$};

\node[below] at (-2,0) {0}; 
\node[below] at (2,0) {1};
\draw[line width=2pt] (-2,0) -- (2,0);
\end{scope}

\begin{scope}[yshift= -3cm]
\draw(-2.5,0) -- (4,0);
\draw[directed] (1,0.7) circle (0.7);
\draw[directed] (3.5,0) arc (0:180:0.8)node[pos=0.45,above, yshift=1pt]{$f_2^m$};

\node[left] at (0.3,0.7) {$f_1^m$};
\node[below] at (-2,0) {0}; 
\node[below] at (2,0) {1};
\draw[line width=2pt] (-2,0) -- (2,0);
\end{scope}

\begin{scope}[xshift =7cm,  yshift= -3cm]
\draw(-2.5,0) -- (4,0);
\draw[directed] (1,0.7) circle (0.7);

\draw[directed] (3.5,0.7) circle (0.7);

\node[left] at (0.3,0.7) {$f_1^m$};
\node[left] at (2.8,0.7) {$f_2^m$};
\node[below] at (-2,0) {0}; 
\node[below] at (2,0) {1};
\draw[line width=2pt] (-2,0) -- (2,0);
\end{scope}

\end{tikzpicture}
\caption{Configurations of the antiparallel pairs of $\F_m$.}\label{mnb}
\end{figure}
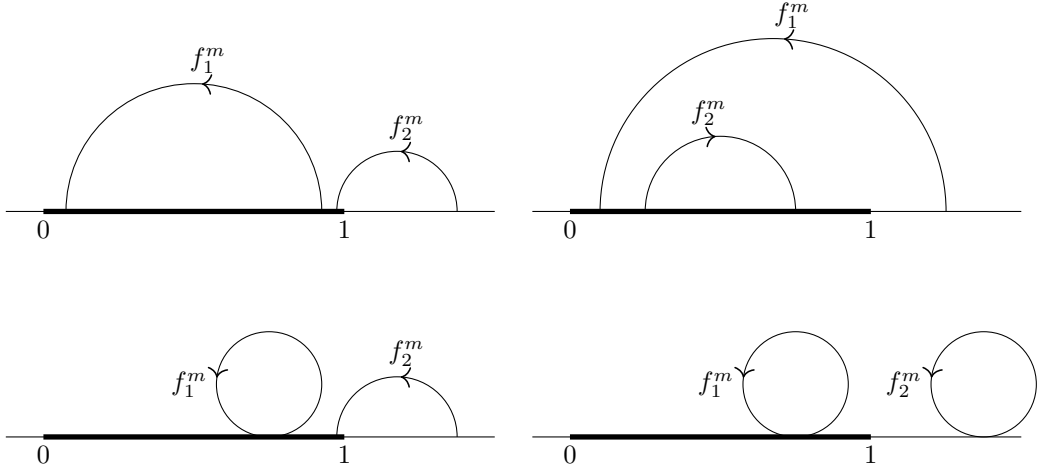

Observe that Theorem~\ref{pr} implies that the principle component is isolated in the parameter space $\HY$, in the sense that $\overline{H_P}$ is disjoint from the closure of all non-principle components. When combined with Theorem~\ref{cl}, however, it implies that the principal component is also isolated from all other components of the parameter space $\SD$, which is a significantly stronger statement.

We are now ready to prove Theorem~\ref{cl}.

\begin{proof}[Proof of Theorem~\ref{cl}]
Suppose that $(\F_m)$ is a sequence in $\SD\setminus \overline{H_P}$ that converges to some $\F\in\PSL^N$. We need to prove that $\F$ lies in $\SDP$. Recall that Lemma~\ref{mj} implies that $N$-tuples that have rank one lie in the closure of the principal component. Therefore $\F_m$ cannot have rank one. So due to Theorem~\ref{pr}, $\F$ cannot lie in the closure of the principal component, and we only need to show that it lies in $\SD$; that is, $\SF$ is semidiscrete and inverse-free. Assume, towards a contradiction, that there exists a sequence of transformations in $\SF$ that converges to the identity. Then, from the convergence of $\F_m$ to $\F$, we can find $f_m\in\langle \F_m \rangle$, for all $m\geq0$, such that $(f_m)$ converges to the identity as $m\to\infty$. By passing on to a subsequence if necessary, we assume that $\alpha(f_m)$, where here we denote by $\alpha(f_m)$ either the attracting or the unique fixed point of $f_m$, converges to some point $\alpha\in\overline{\mathbb{R}}$.\\
Suppose that for all $m$ there exist ordinates $g_m$ of $\F_m $ such that $f_m$ and $g_m$ are antiparallel. We may assume that $(g_m)$ converges to some ordinate $g$ of $\F$. Then J{\o}rgensen's inequality for semigroups, Theorem~\ref{jo}, implies that 
\begin{equation}\label{p}
\lvert \lvert\mathrm{tr}(f_m)\rvert^2 -4\rvert  + \lvert \lvert \mathrm{tr}([f_m,g_m])\rvert -2 \rvert\geq 1.
\end{equation}
But because $f_m$ converges to the identity and $g_m$ to $g$, the commutator $[f_m,g_m]=f_mg_mf_m^{-1}g_m^{-1}$ converges to the identity. So the left-hand side of inequality \eqref{p} converges to zero, as $m\to\infty$, which is a contradiction.\\
Therefore, $f_m$ is not antiparallel with any other ordinate of $\F_m$, for any $m$. So, since $\F_m$ does not have rank one, for every $m\in\mathbb{N}$ we can find an ordinate $h_m$ of $\F_m$, such that $f_m$ and $h_m$ map an open interval $I_m$ inside itself and $\LB(\F_m)\cap I_m\neq \emptyset$. After a suitable conjugation, we can assume that $h_m$ and $I_m$ have the following properties: $h_m(x_m)=x_m$, for some point $x_m$ in $\overline{\mathbb{R}}$ with $x_m<\alpha(f_m)$, and $I_m=(x_m,\alpha(f_m))$. We can also assume that $(h_m)$ converges to some ordinate $h$ of $\F$, and $x_m$ converges to some $x\leq \alpha$. Then, since $(f_m)$ converges to the identity and $h_m$ lies in a bounded subset of $\PSL$, we have that $f_m(x_m)\leq h_m(\alpha(f_m))$, for all $m$ large enough. Hence, we can apply Lemma~\ref{ls-inter} to the transformations $f_n$ and $h_n$ in order to obtain that $\LF(f_m,h_m) =[x_m, \alpha(f_m)]=\overline{I_m}$, for all $m$ large enough. But this implies that $\emptyset\neq\LB(\F_m)\cap I_m\subseteq\LB(\F_m)\cap\LF(\F_m)^{\mathrm{o}}$. Lemma~\ref{nonsd} is then applicable to the $N$-tuple $\F_m$ and yields that either it is a discrete group or it is not semidiscrete, which is a contradiction.\\
In conclusion, contrary to our assumption, $\langle \F\rangle$ is semidiscrete and inverse-free.
\end{proof}

We now study $N$-tuples on the boundary of non-principal components. Recall that such $N$-tuples are semidiscrete and inverse-free due to Theorem~\ref{cl}, but we can also prove that they have finite rank (see Theorem~\ref{npsch} to follow).

We start by defining the `cores" of an $N$-tuple, a concept introduced by Avila, Bochi and Yoccoz in \cite[Section 2.4]{AvBoYo2010}.

\begin{definition}
Let $\F\in\PSL^N$. The complement of the union of the connected components of $\overline{\mathbb{R}}\setminus\LF(\F)$ that intersect $\LB(\F)$ will be called \emph{the forward core} of $\F$ and will be denoted by $C^+(\F)$. Similarly, we define the \emph{backward core} $C^-(\F)$ of $\F$ to be the complement of the union of the connected components of $\overline{\mathbb{R}}\setminus\LB(\F)$ that intersect $\LF(\F)$. 
\end{definition}

Note that for any $\F\in\PSL^N$, because $\LF(\F)$ and $\LB(\F)$ are closed sets, the cores $C^+(\F)$ and $C^-(\F)$ are also closed. Also, for a generic $N$-tuple $\F$, the cores $C^+(\F)$ and $C^-(\F)$ may be trivial subsets of the extended real line. If $\F$ has finite rank, however, then we have the next lemma. Observe that for any $\F\in\PSL^N$ the components of $C^+(\F)$ and $C^-(\F)$ appear in alternating order around $\overline{\mathbb{R}}$. Therefore, $C^+(\F)$ is finitely connected if and only if $C^-(\F)$ is finitely connected. See Figure~\ref{corefig} for examples of the forward and backward cores of a pair in $\PSL^2$ and a triplet in $\PSL^3$.

\begin{lemma}\label{corelemma}
Suppose that $\F\in\SD$ has finite rank. The forward and backward cores of $\F$ satisfy the following properties:
\begin{enumerate}
\item $C^+(\F)\cap C^-(\F)=\partial C^+(\F)\cap \partial C^-(\F)$;
\item $\partial C^+(\F)\subset \LF(\F)$ and $\partial C^-(\F)\subset \LB(\F)$;
\item $\LF(\F)\subseteq C^+(\F)$ and $\LB(\F)\subseteq C^-(\F)$;
\item $C^+(\F)$ and $C^-(\F)$ are finitely connected;
\item $f(C^+(\F))\subseteq C^+(\F)$ and $f^{-1}(C^-(\F))\subseteq C^-(\F)$, for all $f\in\SF$.
\end{enumerate}
\end{lemma}

\begin{proof}
Since $\F$ has finite rank there exists a finite union $U$ of open intervals, with disjoint closures, that is mapped strictly inside itself under $\SF$. So, $\LF(\F)$ lies in $\overline{U}$ and $\LB(\F)$ lies in the complement of $U$. Using similar arguments, we can show that there exists a finite union of open intervals $V$, disjoint from $U$, such that $\SF^{-1}$ maps $V$ strictly inside itself. Thus $C^+(\F)\subseteq\overline{U}$, and $C^-(\F)\subseteq \overline{V}\subseteq U^c$. Parts \textit{(1)} to \textit{(4)} follow easily from these facts.\\
We are now going to show that $C^+(\F)$ is mapped inside itself by $\SF$; the proof for $C^-(\F)$ is similar. Let $I$ be a connected component of $C^+(\F)$. Suppose that there exists an ordinate $f$ of $\F$ such that $f(I)$ is not contained in $C^+(\F)$, and note that parts \textit{(2)} and \textit{(3)} imply that the endpoints of $I$ are mapped into $C^+(\F)$. Hence, there exist points $x\in I$ and $y\in\LB(\F)$ such that $f(x)=y$. Therefore $x=f^{-1}(y)\in\LB(\F)$, which contradicts parts \textit{(1)} and \textit{(3)} of this lemma. So all connected components of $C^+(\F)$ are mapped into $C^+(\F)$ by each ordinate of $\F$, and so we obtain the desired result.
\end{proof}

\begin{figure}[h]
\centering
\begin{tikzpicture}[scale=3]

\begin{scope}[ decoration={
    markings,
    mark=at position 0.7 with {\arrow[scale=1.5]{To}}}
    ] 
\draw (0.,0.) circle (1.cm);
\draw[postaction={decorate}] (0.8140919207904149,0.5807360368564817)-- (-0.8140919207904148,-0.5807360368564818);
\draw[postaction={decorate}] (-0.8140919207904148,0.5807360368564818)-- (0.8140919207904149,-0.5807360368564817);
\draw [|-|,shift={(0,0)},line width=1.pt,xcol, dashed]  plot[domain=3.76122517309801:4.299286004976793,variable=\t]({1.*0.931099836939116*cos(\t r)+0.*0.931099836939116*sin(\t r)},{0.*0.931099836939116*cos(\t r)+1.*0.931099836939116*sin(\t r)});
\draw [|-|,shift={(0.,0.)},line width=1.pt,xcol, dashed]  plot[domain=3.76122517309801:4.299286004976793,variable=\t]({-1.*0.931099836939116*cos(\t r)+0.*0.931099836939116*sin(\t r)},{0.*0.931099836939116*cos(\t r)+1.*0.931099836939116*sin(\t r)});
\draw [|-|,shift={(0.,0.)},line width=1.pt, ycol, dashed]  plot[domain=3.76122517309801:4.299286004976793,variable=\t]({1.*0.931099836939116*cos(\t r)+0.*0.931099836939116*sin(\t r)},{0.*0.931099836939116*cos(\t r)+-1.*0.931099836939116*sin(\t r)});
\draw [|-|,shift={(0.,0.)},line width=1.pt, ycol, dashed]  plot[domain=3.76122517309801:4.299286004976793,variable=\t]({-1.*0.931099836939116*cos(\t r)+0.*0.931099836939116*sin(\t r)},{0.*0.931099836939116*cos(\t r)+-1.*0.931099836939116*sin(\t r)});
\draw [|-|,shift={(0.,0.)},line width=1.5pt, xcol]  plot[domain=3.76122517309801:5.663552787671369,variable=\t]({1.*1.068900163060884*cos(\t r)+0.*1.068900163060884*sin(\t r)},{0.*1.068900163060884*cos(\t r)+1.*1.068900163060884*sin(\t r)});
\draw [|-|,shift={(0.,0.)},line width=1.5pt, ycol]  plot[domain=3.76122517309801:5.663552787671369,variable=\t]({1.*1.068900163060884*cos(\t r)+0.*1.068900163060884*sin(\t r)},{0.*1.068900163060884*cos(\t r)+-1.*1.068900163060884*sin(\t r)});

\end{scope}

\begin{scope}[xshift = 3cm, decoration={
    markings,
    mark=at position 0.7 with {\arrow[scale=1.5]{To}}}
    ] 
    
\draw (0.,0.) circle (1.cm);
\draw[,postaction={decorate}] (0.,-1.) -- (0.,1.);
\draw[postaction={decorate}] (0.8140919207904149,0.5807360368564817)-- (-0.8140919207904148,-0.5807360368564818);
\draw[postaction={decorate}] (-0.8140919207904148,0.5807360368564818)-- (0.8140919207904149,-0.5807360368564817);
\draw [|-|,shift={(0.,0.)},line width=1.pt,xcol, dashed]  plot[domain=1.405165092977344:1.736427560612449,variable=\t]({1.*0.9310998369391159*cos(\t r)+0.*0.9310998369391159*sin(\t r)},{0.*0.9310998369391159*cos(\t r)+1.*0.9310998369391159*sin(\t r)});
\draw  [|-|,shift={(0.,0.)},line width=1.pt,xcol, dashed]    plot[domain=3.5000434979754327:4.299286652380387,variable=\t]({1.*0.934246755792788*cos(\t r)+0.*0.934246755792788*sin(\t r)},{0.*0.934246755792788*cos(\t r)+1.*0.934246755792788*sin(\t r)});
\draw [|-|,shift={(0.,0.)},line width=1.pt,xcol, dashed]  plot[domain=3.5000434979754327:4.299286652380387,variable=\t]({-1.*0.934246755792788*cos(\t r)+0.*0.934246755792788*sin(\t r)},{0.*0.934246755792788*cos(\t r)+1.*0.934246755792788*sin(\t r)});
\draw [|-|,shift={(0.,0.)},line width=1.pt, ycol, dashed]   plot[domain=3.5000434979754327:4.299286652380387,variable=\t]({1.*0.934246755792788*cos(\t r)+0.*0.934246755792788*sin(\t r)},{0.*0.934246755792788*cos(\t r)+-1.*0.934246755792788*sin(\t r)});
\draw  [|-|,shift={(0.,0.)},line width=1.pt, ycol, dashed]  plot[domain=3.5000434979754327:4.299286652380387,variable=\t]({-1.*0.934246755792788*cos(\t r)+0.*0.934246755792788*sin(\t r)},{0.*0.934246755792788*cos(\t r)+-1.*0.934246755792788*sin(\t r)});
\draw [|-|,shift={(0.,0.)},line width=1.pt,ycol, dashed]  plot[domain=4.546757746567137:4.878020214202242,variable=\t]({1.*0.9310998369391159*cos(\t r)+0.*0.9310998369391159*sin(\t r)},{0.*0.9310998369391159*cos(\t r)+1.*0.9310998369391159*sin(\t r)});
\draw [|-|,shift={(0.,0.)},line width=1.5pt,xcol]  plot[domain=3.5000434979754327:4.299286652380387,variable=\t]({1.*1.0689001630608839*cos(\t r)+0.*1.0689001630608839*sin(\t r)},{0.*1.0689001630608839*cos(\t r)+1.*1.0689001630608839*sin(\t r)});
\draw  [|-|,shift={(0.,0.)},line width=1.5pt,xcol]   plot[domain=3.5000434979754327:4.299286652380387,variable=\t]({-1.*1.0689001630608839*cos(\t r)+0.*1.0689001630608839*sin(\t r)},{0.*1.0689001630608839*cos(\t r)+1.*1.0689001630608839*sin(\t r)});
\draw [|-|,shift={(0.,0.)},line width=1.5pt,ycol]  plot[domain=3.5000434979754327:4.299286652380387,variable=\t]({1.*1.0689001630608839*cos(\t r)+0.*1.0689001630608839*sin(\t r)},{0.*1.0689001630608839*cos(\t r)+-1.*1.0689001630608839*sin(\t r)});
\draw [|-|,shift={(0.,0.)},line width=1.5pt,ycol]   plot[domain=3.5000434979754327:4.299286652380387,variable=\t]({-1.*1.0689001630608839*cos(\t r)+0.*1.0689001630608839*sin(\t r)},{0.*1.0689001630608839*cos(\t r)+-1.*1.0689001630608839*sin(\t r)});
\draw [|-|,shift={(0.,0.)},line width=1.5pt,ycol]  plot[domain=4.5467577465671365:4.878020214202243,variable=\t]({1.*1.0662157541469486*cos(\t r)+0.*1.0662157541469486*sin(\t r)},{0.*1.0662157541469486*cos(\t r)+1.*1.0662157541469486*sin(\t r)});
\draw [|-|,shift={(0.,0.)},line width=1.5pt,xcol]  plot[domain=4.5467577465671365:4.878020214202243,variable=\t]({1.*1.0662157541469486*cos(\t r)+0.*1.0662157541469486*sin(\t r)},{0.*1.0662157541469486*cos(\t r)+-1.*1.0662157541469486*sin(\t r)});
\end{scope}

\end{tikzpicture}
\caption{Examples of cores. The dashed lines indicate the limit sets, while the continuous lines the cores.}\label{corefig}
\end{figure}
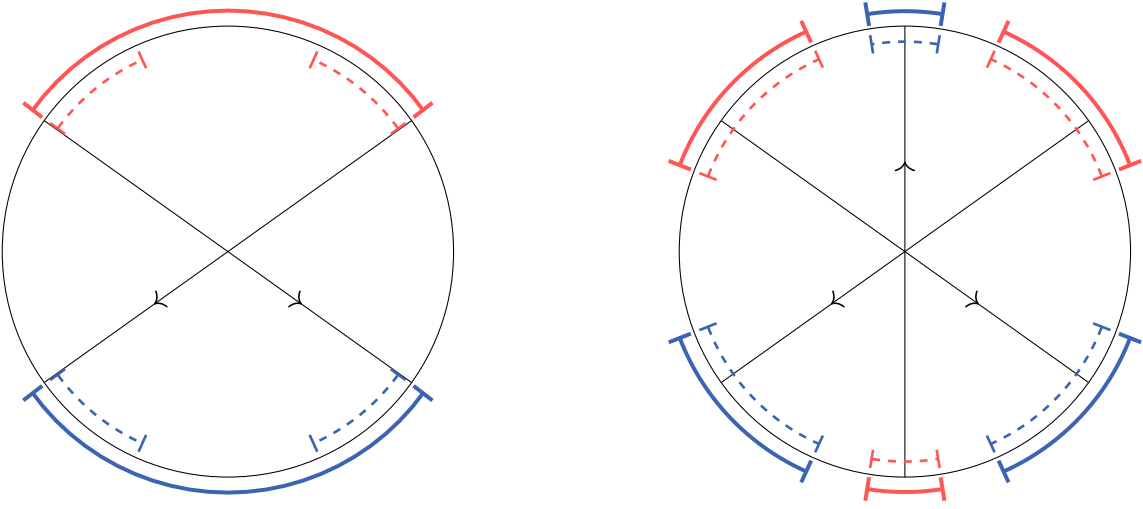

It is interesting to note that Lemma \ref{corelemma} shows that if $\F$ has rank $m$ then the sets $C^+(\F)$ and $C^-(\F)$ are $m$-connected. 

Using the properties of the cores we prove that points on the boundary of non-principal components have finite rank (recall that they are semidiscrete and inverse-free by Theorem~\ref{cl}). This can be thought of as an extension of \cite[Proposition 4.9]{AvBoYo2010}, which states that if $\F$ lies on the boundary of a non-principal component then it is inverse-free.

\begin{theorem}\label{npsch}
Suppose that $\F$ lies on the boundary of a non-principal component $H$ of $\HY$. Then the rank of $\F$ is finite and greater than one.
\end{theorem}

\begin{proof}
Observe that $\F$ lies in $\partial H\subset\SDP$ and thus it is semidiscrete and inverse-free due to Theorem~\ref{cl}. Also, because $\F$ does not lie in the closure of the principal component, if $\F$ had finite rank, that would have to be greater than one. Let $\F_n\in H$ be a sequence of $N$-tuples with $\F_n\to \F$ and define $C_n^+=C^+(\F_n)$ and $C_n^-=C^-(\F_n)$. Under these assumptions \cite[Proposition 4.13]{AvBoYo2010} states that $C_n^+$ and $C_n^-$ converge, in the Hausdorff metric, to some sets $C^+$ and $C^-$ that, by continuity, are finitely connected, nonempty subsets of $\overline{\mathbb{R}}$. We are going to prove that $C^+$ is mapped strictly inside itself by each ordinate of $\F$, and that it is a union of nontrivial closed intervals, none of which are singletons.\\
First, take $x\in C^+$ and let $f$ be an ordinate of $\F$. There exist sequences of points $(x_n)$ in $C_n^+$ and ordinates $(f_n)$ of $\F_n$, such that $x_n\to x$ and $f_n\to f$, as $n\to \infty$. Conjugating by a M\"obius transformation we can assume that $x=1$ and $f(z)=\lambda z$, for some positive $\lambda\in\mathbb{R}$. Also, without loss of generality, we assume that $x_n\leq1$, for all $n$. Applying the mean value theorem to the interval $(x_n,1)$, for some $n$, we obtain that there exists $\xi_n\in (x_n,1)$, such that 
\[
|f_n(1)-f_n(x_n)|= |f_n'(\xi_n)| |1-x_n|.
\]
But $f_n'(\xi_n)$ converges to $\lambda$, and so $f_n(x_n)$ converges to the point $f(1)$, as $n\to\infty$. Since $\F_n$ lie in a non-principal hyperbolic component, they have rank greater than one. Hence, Lemma~\ref{corelemma} implies that $C_n^+$ is forward invariant under $\langle \F_n \rangle$. So the points $f_n(x_n)$ lie in $C_n^+$, for all $n$. Thus, by the convergence of $C_n^+$ to $C^+$, we have that $f(1)\in C^+$. We conclude that $C^+$ is mapped inside itself by each element of $\langle \F \rangle$.\\
Assume now that $C^+$ has a connected component $\{x\}$, for some $x\in\overline{\mathbb{R}}$. The second part of Lemma~\ref{corelemma} implies that $x$ lies in $\LF(\F)$. Because $\F$ does not lie in the closure of the principal component, it is non-elementary. So Lemma~\ref{nonel} implies that $\LF(\F)$ is the closure of all attracting fixed points of $\SF$. Thus we have that $x=\alpha(h)$, for some hyperbolic transformation $h\in\SF$. But $C^+$ is mapped inside itself by $\SF$ and so $C^+=\{x\}$. Thus $\F$ has to lie in $\overline{H_P}$, which is a contradiction.\\
Using similar arguments we can prove that $C^-$ is mapped inside itself by each element of $\SF^{-1}$, and does not contain any connected components that are singletons. Hence, the set $C^+$ cannot be the extended real line, because otherwise the convergence of $C_n^+$ and $C_n^-$ would imply that $C^-$ is a finite subset of $\overline{\mathbb{R}}$.\\
We will now prove that $\SF$ maps $C^+$ strictly inside itself. Suppose that there exists an ordinate $g$ of $\F$ that fixes $C^+$. Since $C^+$ cannot be empty, the extended real line or have components that are singletons, $C^+$ has to be an interval. Then $\SF$ lies in $\mathcal{M}(C^+)$, which by Lemma~\ref{mj} implies that $\F$ lies in the closure of the principal component, a contradiction.\\
We conclude that $C^+$ is a union of closed, disjoint intervals that is mapped strictly inside itself by $\SF$, and thus $\F$ has finite rank.
\end{proof}

It is not known whether the inverse of Theorem~\ref{npsch} holds (see Question \ref{imp2} in the final section). We do, however, have the following lemma about the structure of the limit sets of $N$-tuples with finite rank. This result is, in some sense, an analogue of \cite[Lemma 4.12]{AvBoYo2010}, and the proof of the two lemmas are similar.

\begin{lemma}\label{poi}
Suppose that $\F=(f_1,f_2,\dots ,f_N)\in\SD$ has finite rank and $\LF(\F)\cap\LB(\F)\neq\emptyset$. Then for each point $x\in \LF(\F)\cap\LB(\F)$ there exist points $a,b\in\LF(\F)\cap\LB(\F)$ and transformations $f,g\in \SF\cup\{\mathrm{Id}\}$ such that $x=f(a)=g^{-1}(b)$. In addition, $a$ is either an attracting or parabolic fixed point and $b$ is either a repelling or parabolic fixed point of transformations in $\SF$.
\end{lemma}

\begin{proof}
Because $\F$ has finite rank, Lemma~\ref{corelemma} implies that $C^+(\F)$ and $C^-(\F)$ are finitely connected and $\LF(\F)\cap\LB(\F)=\partial C^+(\F)\cap \partial C^-(\F)$. Thus the intersection $\LF(\F)\cap\LB(\F)$ is finite. Note that if $C^+(\F)$ or $C^-(\F)$ is a single interval (i.e. $\F$ has rank one) then the result is obvious (see, for example, \cite[Proposition 4]{Yo2004}).\\
Let $x\in \LF(\F)\cap\LB(\F)$. Since $x$ lies in the forward limit set of $\SF$, there exists a sequence $(k_n)$ in $\SF$ such that $k_n(i)\to x$, as $n\to \infty$. Because the semigroup $\SF$ is generated by a finite collection of M\"obius transformations, by passing to a subsequence if necessary, we can write $k_n$ as $k_n=f_ih_n$, for some ordinate $f_i$ of $\F$ and some sequence $(h_n)$ in $\SF$. Then $h_n=f_i^{-1}k_n$, for all $n$, and so $(h_n(i))$ converges to $f_i^{-1}(x)\in\overline{\mathbb{R}}$. Thus, because $h_n\in\SF$, for all $n$, the point $f_i^{-1}(x)$ lies in the forward limit set of $\SF$. But $x$ also lies in $\LB(\F)$, which is a backward invariant set, and so $f_i^{-1}(x)$ also lies in the backward limit set of $\SF$.\\
Hence, for every $x\in\LF(\F)\cap\LB(\F)$ there exists an ordinate $f_i$ of $\F$ such that $f_i^{-1}(x)\in\LF(\F)\cap\LB(\F)$. Arguing similarly, we can show that there exists an ordinate $f_j$ of $\F$ such that $f_j(x)\in\LF(\F)\cap\LB(\F)$. Iterating these processes, we obtain sequences $(i_n),(j_n)\subset \{1,2,\dots , N\}$ such that the points 
\begin{align*}
&x_n=f^{-1}_{i_n}f^{-1}_{i_{n-1}}\dots f^{-1}_{i_1}(x) &\text{and}& &y_n=f_{j_n}f_{j_{n-1}}\dots f_{j_1}(x),
\end{align*}
lie in $\LF(\F)\cap\LB(\F)$, for all $n$. Since $\LF(\F)\cap\LB(\F)$ is finite, there exist positive integers $k,m$ and $\lambda,\nu$ such that $x_m=x_{m+k}$ and $y_\nu=y_{\nu+\lambda}$. Hence, the transformation $f^{-1}_{i_{m+k}}f^{-1}_{i_{m+k-1}}\dots f^{-1}_{i_{m+1}}$ fixes the point $x_m=f^{-1}_{i_m}f^{-1}_{i_{m-1}}\dots f^{-1}_{i_1}(x)$ and similarly $f_{j_{\nu+\lambda}}f_{j_{\nu+\lambda-1}}\dots f_{j_{\nu+1}}$ fixes $y_\nu=f_{j_\nu}f_{j_{\nu-1}}\dots f_{j_1}(x)$. Also
\[
x=f_{i_1}f_{i_2}\dots f_{i_m}(x_m)=f^{-1}_{j_1}f^{-1}_{j_2}\dots f^{-1}_{j_\nu}(y_\nu).
\]
We are now going to show that $x_m$ is either an attracting or a parabolic fixed point; the proof for $y_\nu$ is similar. For simplicity we write $g_n=f^{-1}_{i_{m+n}}f^{-1}_{i_{m+n-1}}\dots f^{-1}_{i_{m+1}}$, for $n=1,2,\dots, k$. Recall that the point $x_m$ lies on the boundary of $C^+$, and let $D$ be the unique connected component of $C^+$ that contains $x_m$. In addition, $g_n(x_m)\in \LF(\F)\cap\LB(\F)$ for all $n=1,2,\dots, k$, and we define $D_n$ to be the component of $C^+$ that contains $g_n(x_m)$. Observe that $g_k(x_m)=x_m$ and so $D_k=D$. Because the forward core of $\F$ is mapped inside itself by all $f\in\SF$, we have that each such transformation $f$ maps the interior of $C^+$ into itself. Thus, because $g_n^{-1}$ lies in $\SF$, we have that $g_n^{-1}(D_n)\subseteq D$ for all $n=1,2,\dots, k$. Since M\"obius transformations are conformal, this is equivalent to $D_n\subseteq g_n(D)$, for all $n=1,2,\dots, k$. In particular, we have that $D_k=D\subseteq g_k(D)$. Suppose that $g_k(D)=D_k=D$. Then the forward core $C^+$ is the interval $D$ and, as we mentioned in the beginning of the proof, the result is obvious. So, assume that $D\subsetneq g_k(D)$. Recalling that $g_n$ fixes $x_m$, which is an endpoint of $D$, we conclude that $x_m$ is either the repelling (or the unique) fixed point of $g_k$, and hence the attracting (or unique) fixed point of $g_k^{-1}=f_{i_{m+1}}f_{i_{m+2}}\dots f_{i_{m+k}}$.
\end{proof}

\section{A counterexample to Question~\ref{Q}}\label{counter}

We are now ready to present our example that answers Question~\ref{Q} in the negative. Recall from the indtroduction that $\EL$ is the set of all $\F\in\PSL^N$ such that $\SF$ contains either an elliptic transformation or the identity function. Also recall that the semidiscrete and inverse-free locus is disjoint from the elliptic locus.

We start by presenting an example for the case $N=3$, which we will then extend to any $N\geq 3$.

\begin{example}\label{f0}
Suppose that $N=3$ and consider the following hyperbolic transformations: $g_1(z)=2z+1$, $g_2(z)=\tfrac{1}{3}z$ and $g_3(z)=5z-4$ (see Figure~\ref{3} for the axes of the transformations $g_1,g_2$ and $g_3$). Observe that the triplet $\F_0=(g_1,g_2,g_3)$ is elementary, since $g_1,g_2$ and $g_3$ all fix the point at infinity. It is easy to check that any element in $\langle \F_0 \rangle$ is of the form $\lambda z+ \kappa$, where $\lambda=2^l3^{-m}5^n$, for some $l,m,n\in\mathbb{N}\cup\{0\}$, and some real number $\kappa$. Note that $l,m$ and $n$ cannot all be 0 simultaneously, and so because $2,3$ and $5$ are prime numbers, $\lambda$ cannot be 1. Thus all the transformations in $\langle \F_0 \rangle$ are hyperbolic, which implies that the semigroup $\langle \F_0 \rangle$ does not contain any elliptic transformations or the identity. Hence, the triplet $\F_0$ lies in the complement of $\EL$ in $\PSL^3$.\\
We claim that $\F_0$ is not semidiscrete. Note that due to Lemma~\ref{limitset} the limit set of the pair $(g_1,g_2)$ is the closed interval $[0,\infty]$. Also, the repelling fixed point of $g_3$ is $1$ which lies in the interior of $[0,\infty]=\LF(g_1,g_2)$. Hence, Lemma~\ref{ls-inter} is applicable and yields that either $\langle \F_0 \rangle$ is a discrete group, or $\F_0$ is not semidiscrete. But $\F_0$ is inverse-free, so the semigroup it generates cannot be a group; this proves our claim. \qed
\end{example}

\begin{figure}[ht]
\centering
\begin{tikzpicture}

\begin{scope}[scale = 3, decoration={
    markings,
    mark=at position 0.6 with {\arrow{To}}}]
\draw [line width=0.4pt] (0.,0.) circle (1.cm);
\draw [line width=.4pt,postaction={decorate}] (0.,1.)-- (0.,-1.);
\end{scope}

\begin{scope}[scale = 3, decoration={
    markings,
    mark=at position 0.4 with {\arrow{To}}}]
\draw [shift={(-1.6159487086371378,1.)},line width=0.4pt,postaction={decorate}]  plot[domain=-1.1083025081473314:0.,variable=\t]({1.*1.6159487086371378*cos(\t r)+0.*1.6159487086371378*sin(\t r)},{0.*1.6159487086371378*cos(\t r)+1.*1.6159487086371378*sin(\t r)});
\draw [shift={(1.6159487086371385,1.)},line width=0.4pt,postaction={decorate}]  plot[domain=4.249895161737125:3.141592653589793,variable=\t]({1.*1.6159487086371385*cos(\t r)+0.*1.6159487086371385*sin(\t r)},{0.*1.6159487086371385*cos(\t r)+1.*1.6159487086371385*sin(\t r)});

\node[above] at (0,1) {$\infty$};
\node[below] at (0,-1) {$0$};
\node[ below left] at (-0.89, -0.45) {$-1$};
\node[below right] at (0.89, -0.45) {$1$};
\node[left] at (0,-.2) {$g_2$};
\node[left] at (-0.37,0.05) {$g_1$};
\node[right] at (0.37,0.05) {$g_3$};

\end{scope}
\end{tikzpicture}
\caption{The 3-tuple $\F_0=(g_1,g_2,g_3)$}\label{3}
\end{figure}
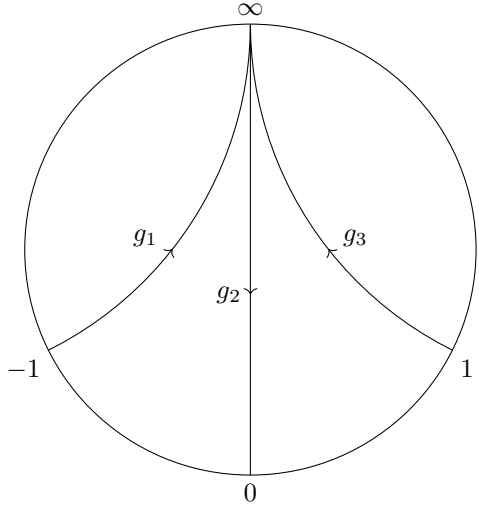

In order to show that Example~\ref{f0} answers Question~\ref{Q} in the negative, we need to prove that $\F_0$ lies in the complement of $\overline{\HY}$. Note that because $\HY\subset \SD$, it suffices to prove that $\F_0$ lies in the complement of $\overline{\SD}$.\\
Recall that the set $\SDP$ is closed, as stated by Theorem~\ref{cl}. Therefore, because $ \F_0$ is not semidiscrete, if $\F_0$ were to lie in $\overline{\SD}$, it would have to lie in the closure of the principal component. The following lemma tells us that this cannot happen, showing that $\F_0\in\overline{\HY}^c$.

\begin{lemma}
The triplet $\F_0$ in Example~\ref{f0} lies in the complement of $\overline{H_P}$.
\end{lemma}

\begin{proof}
It suffices to show that there exists an open neighbourhood of $\F_0$ in $\PSL^3$ that is disjoint from the principal component $H_P$. Take $\epsilon>0$ and define the open intervals $U_{-1}=\left(-1-\epsilon, -1+\epsilon\right)$, $U_0=\left(-\epsilon,\epsilon\right)$ and $U_1=\left(1-\epsilon, 1+\epsilon\right)$. Also define $U_{\infty}=[\infty,-1/\epsilon)\cup (1/\epsilon,\infty]$, which is a neighbourhood of $\infty$ in $\overline{\mathbb{R}}$. We can choose $\epsilon$ small enough such that the sets $U_{-1},U_0,U_1$ and $U_{\infty}$ have pairwise disjoint closures. Recall that the set of hyperbolic transformations in $\PSL$ is open. So there exists an open neighbourhood $D$ of $\F_0$ in $\PSL^3$, such that if $(f_1,f_2,f_3)$ lies in $D$, then $f_1, f_2$ and $f_3$ are hyperbolic transformations. We can also choose $D$ small enough so that for all $(f_1,f_2,f_3)\in D$ we have $\beta(f_1)\in U_{-1},\, \alpha(f_2)\in U_0,\, \beta(f_3)\in U_1$ and $\alpha(f_1),\beta(f_2),\alpha(f_3)\in U_{\infty}$ (see Figure~\ref{nhd}). Hence, by the configuration of the axes of $f_1,f_2$ and $f_3$ we see that $(f_1,f_2,f_3)$ cannot have rank one, and this implies that $D$ does not intersect $H_P$, as required.
\end{proof}

\begin{figure}[ht]
\centering
\begin{tikzpicture}

\draw (-5,0) -- (5,0);
\draw[line width=2pt] (-5,0) -- (-4,0);
\draw[line width=2pt] (4,0) -- (5,0);
\draw[line width=2pt] (-2.25,0) -- (-1.75,0);
\draw[line width=2pt] (-.25,0) -- (.25,0);
\draw[line width=2pt] (1.75,0) -- (2.25,0);

\draw[directed] (-2,0) arc (0:150:2)node[pos=0.45,above, yshift=1pt]{$f_1$};
\draw[directed] (2,0) arc (180:30:2)node[pos=0.45,above, yshift=1pt]{$f_3$};
\draw[directed] (0,2) -- (0,0) node[pos=0.5,right]{$f_2$};

\node[below] at (-4,0) {$-\frac{1}{\epsilon}$};
\node[below] at (4,0) {$\frac{1}{\epsilon}$};
\node[below] at (-2,0) {$U_{-1}$};
\node[below] at (0,0) {$U_0$};
\node[below] at (2,0) {$U_1$};
\end{tikzpicture}
\caption{An open neighbourhood $U$ of $\F_0$}\label{nhd}
\end{figure}
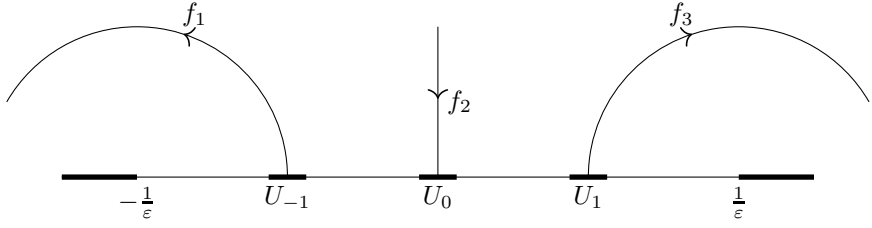

We now extend Example~\ref{f0} to any $N\geq 3$. 
\begin{definition}\label{p}
We define $\mathcal{P}$ to be the set of $N$-tuples that are elementary, inverse-free, do not contain elliptic transformations, but are not semidiscrete.
\end{definition}
First, let us establish that $\mathcal{P}$ is non-empty for all $N\geq3$. If $N=3$, then the triplet $\F_0$ in Example~\ref{f0} lies in $\mathcal{P}$. Suppose that $N>3$. Consider the transformations $g_1(z)=2z+1$, $g_2(z)=\tfrac{1}{3}z$ and $g_3(z)=5z-4$ that were introduced in Example~\ref{f0}. Let $\F=(g_1,g_2,g_3,g_4,\dots,g_N)$, where $g_i$ is a hyperbolic transformation that fixes the point at infinity and does not fix any of the points $-1,0$ and 1, for all $i=4,5,\dots, N$. Hence, for all $i=4,5,\dots,N$, $g_i$ is of the form $g_i(z)=\lambda_i z+\kappa_i$, for some $\lambda_i>0$ and $\kappa_i\in\mathbb{R}$. We also choose $\lambda_i$ to be either $p$ or $\tfrac{1}{p}$, where $p$ is a prime number different from $2,3$ and $5$. So it is easy to see that all elements in $\SF$ are hyperbolic transformations that fix the point infinity, and thus $\F$ is elementary and inverse-free. In addition, because $\langle \F_0 \rangle\subset \SF$, the $\F$ is not semidiscrete. We conclude that the $N$-tuple $\F$ lies in $\mathcal{P}$.

Observe that $\mathcal{P}$ does not intersect the elliptic locus. We now provide the following description for the complement of the semidiscrete and inverse-free locus. Note that because $\HY\subset \SD$, Theorem~\ref{sdc} answers Question~\ref{Q} in the negative for any $N\geq 3$.

\begin{theorem}\label{sdc}
For all $N\geq 3$, we have that $\left(\SDP\right)^c=\EL\cup \overline{H_P}\cup \mathcal{P}$ in $\PSL^N$.
\end{theorem}

Observe that $\EL\cup \overline{H_P}\cup \mathcal{P}$ is not a disjoint union (for example $\EL\cap \overline{H_P}\neq \emptyset$), but due to Theorem~\ref{cl} it is an open subset of $\PSL^N$. In order to prove Theorem~\ref{sdc} we require the following classification of finitely-generated semigroups due to Jacques and Short \cite[Theorem14.1]{JS} (recall Definition~\ref{mjd}).

\begin{theorem}[\cite{JS}]\label{class}
Let $\F$ be a $N$-tuple in $\PSL^N$. Then either 
\begin{enumerate}
\item $\F$ is elementary;
\item $\F$ is semidiscrete;
\item $\langle \F \rangle$ is contained in $\mathcal{M}(J)$, for some $J$; or
\item $\langle \F \rangle$ dense in $\PSL$.
\end{enumerate}
\end{theorem}

We remark that the four classes in Theorem~\ref{class} are not all disjoint. For example, there exist elementary and semidiscrete $N$-tuples which generate semirgoups that lie in $\mathcal{M}(J)$, for some $J$. However, semidiscrete $N$-tuples cannot generate semigroups that are dense in $\PSL$ and vice versa.

\begin{proof}[Proof of Theorem~\ref{sdc}]
The inclusion $\EL\cup \overline{H_P}\cup \mathcal{P}\subseteq \left(\SDP\right)^c$ is obvious from the definitions of $\EL,H_P$ and $\mathcal{P}$. For the converse, let $\F$ be an $N$-tuple in $\left(\SDP\right)^c$, and suppose that $\F\notin \left(\EL\cup \overline{H_P}\right)$. Our task is to show that $\F$ lies in $\mathcal{P}$. Note that because $\F\notin \EL$, the semigroup $\SF$ is inverse-free and does not contain elliptic transformations. Corollary~\ref{sc} implies that $\left(\SD\cup \overline{H_P}\right)^c=\overline{\SD}^c$, and so the semigroup generated by $\F$ is not semidiscrete.  Hence, we only need to show that $\SF$ is elementary.\\
Observe that due to Lemma~\ref{mj}, the semigroup $\SF$ cannot be contained in $\mathcal{M}(J)$ because then it would have to lie in the closure of the principal component. Also, since the set of elliptic transformations of $\PSL$ is open, $\SF$ cannot be dense in $\PSL$ as then $\F$ would lie in $\EL$. Finally, from Theorem~\ref{class}, $\SF$ has to be elementary and thus $\F$ lies in $\mathcal{P}$, as required.
\end{proof}

\section{Questions and future directions}

Even though the introduction of the semidiscrete and inverse-free locus, and Theorem~\ref{cl}, provides us with a general framework with which to study the elusive boundary of the hyperbolic locus, it does introduce several questions about the structure of $\PSL^N$ in general.

The most important question arising from our work is the following updated version of Question~\ref{Q}.

\begin{question}\label{imp}
Is $\SDP$ equal to the closure of $\HY\setminus H_P$ in $\PSL^N$?
\end{question}

A positive answer to Question \ref{imp} would imply that Theorem~\ref{sdc} provides a complete characterisation of the complement of $\overline{\HY}$, thus finally settling \cite[Question 4]{AvBoYo2010} and \cite[Question 4]{Yo2004} of Avila, Bochi and Yoccoz. The inclusion $\overline{\HY\setminus H_P}\subseteq \SDP$ is obvious from the fact that $\HY\setminus H_P$ is contained in the closed set $\SDP$. Hence, the difficulty in Question~\ref{imp} lies in evaluating whether every $N$-tuple in $\SD$ that does not have rank one, lies in the closure of $\HY$.

A first step in answering Question~\ref{imp} would be to focus on $N$-tuples of finite rank. Due to Lemma~\ref{mj}, we only need to consider $N$-tuples that have rank strictly greater than one, since all others are contained in the closure of the principal component. So we ask the following:

\begin{question}\label{imp2}
Let $\F$ be an $N$-tuple that has finite rank greater than one. Does $\F$ lie in the closure of $\HY\setminus H_P$?
\end{question}

Recall Theorem~\ref{npsch}, which states that all $N$-tuples in the closure of a non-principal component have rank greater than one. This, together with Lemma~\ref{poi} seem to indicate that the answer to Question~\ref{imp2} is positive. In order to answer this question in the positive, one would have to perturb the ordinates of the generating $N$-tuple in such a way that the intersection points of the limit sets vanish.

What is important about Question~\ref{imp2} is that a positive answer would imply that points of $\SD$ that lie in the interior of $\overline{\EL}$, if those indeed exist, have to generate semigroups with ``wild" limit sets. However, a positive answer to Question~\ref{imp2} still does not provide an answer to Question~\ref{imp}. Instead, we ask the following:

\begin{question}\label{big}
Suppose that $\F$ is an $N$-tuple in $\SD$. Can we find a sequence $(\F_n)$ in $\SD$ such that $\F_n\to \F$, as $n\to \infty$, and $\F_n $ have finite rank?
\end{question}

Question~\ref{big} is equivalent to examining whether every semigroup generated by a semidiscrete and inverse-free $N$-tuple is the algebraic limit of semigroups generated by $N$-tuples of finite rank (see \cite[Chapter 4]{Ma2007} for the definition of algebraic convergence). A positive answer to Questions \ref{imp2} and \ref{big} would provide a positive answer to Question~\ref{imp}. Also, due to the fact that $\SDP$ is closed, a counterexample to Question~\ref{big} would have to be a semigroup with limit sets that intersect in a convoluted way, and a generating $N$-tuple that is isolated from the hyperbolic components of $\HY$. This would most likely mean that the $N$-tuple has to lie in the interior of $\overline{\EL}$, thus also providing a counterexample to Question~\ref{imp}.


\begin{bibdiv}
\begin{biblist}


\bib{AvBoYo2010}{article}{
   author={Avila, Artur},
   author={Bochi, Jairo},
   author={Yoccoz, Jean-Christophe},
   title={Uniformly hyperbolic finite-valued ${\rm SL}(2,\Bbb R)$-cocycles},
   journal={Comment. Math. Helv.},
   volume={85},
   date={2010},
   number={4},
   pages={813--884},
}

\bib{ASVW2013}{article}{
   author={Avila, Artur},
   author={Santamaria, Jimmy},
   author={Viana, Marcelo},
   author={Wilkinson, Amie},
   title={Cocycles over partially hyperbolic maps},
   journal={Ast\'{e}risque},
   number={358},
   date={2013},
   pages={1--12}
}


\bib{AvVi2010}{article}{
   author={Avila, Artur},
   author={Viana, Marcelo},
   title={Extremal Lyapunov exponents: an invariance principle and
   applications},
   journal={Invent. Math.},
   volume={181},
   date={2010},
   number={1},
   pages={115--189}
}

\bib{BaBeCa1996}{article}{
   author={B\'{a}r\'{a}ny, I.},
   author={Beardon, A. F.},
   author={Carne, T. K.},
   title={Barycentric subdivision of triangles and semigroups of M\"{o}bius
   maps},
   journal={Mathematika},
   volume={43},
   date={1996},
   number={1},
   pages={165--171}
}



\bib{BaKaKo2018}{article}{
   author={B\'{a}r\'{a}ny, Bal\'{a}zs},
   author={K\"{a}enm\"{a}ki, Antti},
   author={Koivusalo, Henna},
   title={Dimension of self-affine sets for fixed translation vectors},
   journal={J. Lond. Math. Soc. (2)},
   volume={98},
   date={2018},
   number={1},
   pages={223--252}
}


\bib{BaPe2007}{book}{
   author={Barreira, Luis},
   author={Pesin, Yakov},
   title={Nonuniform hyperbolicity},
   series={Encyclopedia of Mathematics and its Applications},
   volume={115},
   note={Dynamics of systems with nonzero Lyapunov exponents},
   publisher={Cambridge University Press, Cambridge},
   date={2007},
   pages={xiv+513}
}

\bib{Be1995}{book}{
   author={Beardon, Alan F.},
   title={The geometry of discrete groups},
   series={Graduate Texts in Mathematics},
   volume={91},
   note={Corrected reprint of the 1983 original},
   publisher={Springer-Verlag, New York},
   date={1995},
   pages={xii+337},
}





\bib{BeMi2007}{article}{
   author={Beardon, A. F.},
   author={Minda, D.},
   title={The hyperbolic metric and geometric function theory},
   conference={
      title={Quasiconformal mappings and their applications},
   },
   book={
      publisher={Narosa, New Delhi},
   },
   date={2007},
   pages={9--56},
}

\bib{BoGa2019}{article}{
   author={Bochi, Jairo},
   author={Garibaldi, Eduardo},
   title={Extremal norms for fiber-bunched cocycles},
   journal={J. \'{E}c. polytech. Math.},
   volume={6},
   date={2019},
   pages={947--1004}
}

\bib{BoGo2009}{article}{
   author={Bochi, Jairo},
   author={Gourmelon, Nicolas},
   title={Some characterizations of domination},
   journal={Math. Z.},
   volume={263},
   date={2009},
   number={1},
   pages={221--231}
   }

\bib{BoMo2015}{article}{
   author={Bochi, Jairo},
   author={Morris, Ian D.},
   title={Continuity properties of the lower spectral radius},
   journal={Proc. Lond. Math. Soc. (3)},
   volume={110},
   date={2015},
   number={2},
   pages={477--509}
}


\bib{BoPoSa2019}{article}{
   author={Bochi, Jairo},
   author={Potrie, Rafael},
   author={Sambarino, Andr\'{e}s},
   title={Anosov representations and dominated splittings},
   journal={J. Eur. Math. Soc. (JEMS)},
   volume={21},
   date={2019},
   number={11},
   pages={3343--3414}
}


\bib{Fo2006}{article}{
   author={Forni, Giovanni},
   title={On the Lyapunov exponents of the Kontsevich-Zorich cocycle},
   conference={
      title={Handbook of dynamical systems. Vol. 1B},
   },
   book={
      publisher={Elsevier B. V., Amsterdam},
   },
   date={2006},
   pages={549--580}
}



\bib{FrMaSt2012}{article}{
   author={Fried, David},
   author={Marotta, Sebastian M.},
   author={Stankewitz, Rich},
   title={Complex dynamics of M\"obius semigroups},
   journal={Ergodic Theory Dynam. Systems},
   volume={32},
   date={2012},
   number={6},
   pages={1889--1929},
   }


\bib{Gu1995}{article}{
   author={Gurvits, Leonid},
   title={Stability of discrete linear inclusion},
   journal={Linear Algebra Appl.},
   volume={231},
   date={1995}
}



\bib{HiMa1996}{article}{
   author={Hinkkanen, A.},
   author={Martin, G. J.},
   title={The dynamics of semigroups of rational functions. I},
   journal={Proc. London Math. Soc. (3)},
   volume={73},
   date={1996},
   number={2},
   pages={358--384},
}


\bib{JS}{article}{
   author={Jacques, Matthew},
   author={Short, Ian},
   title={Dynamics of hyperbolic isometries},
   eprint={https://arxiv.org/abs/1609.00576v4}
}

\bib{Ju2012}{article}{
   author={Jungers, Rapha\"{e}l M.},
   title={On asymptotic properties of matrix semigroups with an invariant
   cone},
   journal={Linear Algebra Appl.},
   volume={437},
   date={2012},
   number={5},
   pages={1205--1214}
}

\bib{Ka2011}{article}{
   author={Kalinin, Boris},
   title={Liv\v{s}ic theorem for matrix cocycles},
   journal={Ann. of Math. (2)},
   volume={173},
   date={2011},
   number={2},
   pages={1025--1042}
}


\bib{Ka1992}{book}{
   author={Katok, Svetlana},
   title={Fuchsian groups},
   series={Chicago Lectures in Mathematics},
   publisher={University of Chicago Press, Chicago, IL},
   date={1992},
   pages={x+175}
}
   


\bib{Ma2007}{book}{
   author={Marden, A.},
   title={Outer circles},
   note={An introduction to hyperbolic 3-manifolds},
   publisher={Cambridge University Press, Cambridge},
   date={2007},
   pages={xviii+427}
}

\bib{Vi2008}{article}{
   author={Viana, Marcelo},
   title={Almost all cocycles over any hyperbolic system have nonvanishing
   Lyapunov exponents},
   journal={Ann. of Math. (2)},
   volume={167},
   date={2008},
   number={2},
   pages={643--680}
}

\bib{Vi2014}{book}{
   author={Viana, Marcelo},
   title={Lectures on Lyapunov exponents},
   series={Cambridge Studies in Advanced Mathematics},
   volume={145},
   publisher={Cambridge University Press, Cambridge},
   date={2014},
   pages={xiv+202}
}




\bib{Yo2004}{article}{
   author={Yoccoz, Jean-Christophe},
   title={Some questions and remarks about ${\rm SL}(2,\bold R)$ cocycles},
   conference={
      title={Modern dynamical systems and applications},
   },
   book={
      publisher={Cambridge Univ. Press, Cambridge},
   },
   date={2004},
   pages={447--458},
}



\end{biblist}
\end{bibdiv}
\end{document}